\documentclass[a4paper,10pt,reqno]{article}
\usepackage{amsmath}                        
\usepackage{amsthm}                         
\usepackage{amssymb}                        
\usepackage{amsfonts}                       
\usepackage{latexsym}                       
\usepackage{setspace}						
\usepackage[colorlinks=true,linkcolor=blue,citecolor=blue]{hyperref}
\usepackage{indentfirst}

\newtheorem{theorem}{Theorem}[section]
\newtheorem{proposition}[theorem]{Proposition}
\newtheorem{corollary}[theorem]{Corollary}
\newtheorem{lemma}[theorem]{Lemma}

\newtheorem{definition}[theorem]{Definition}
\theoremstyle{remark}
\newtheorem{remark}[theorem]{Remark}

\newcommand{\la}{\langle}
\newcommand{\ra}{\rangle}
\def\R{{\mathbb{R}}}

\def\H{{\vec{H}}}
\def\a{a}
\def\b{b}

\DeclareMathOperator{\tr}{tr}

\title{\bf On Lorentzian surfaces in $\R^{2,2}$}

\author{Pierre Bayard\footnote{bayard@ciencias.unam.mx, Facultad de Ciencias, Universidad Nacional Aut\'onoma de M\'exico, M\'exico}, Victor Patty\footnote{victorp@ifm.umich.mx, Instituto de F\'{\i}sica y Matem\'aticas, U.M.S.N.H., Ciudad Universitaria, CP. 58040 Morelia, Michoac\'an, M\'exico}, Federico S\'anchez-Bringas\footnote{sanchez@unam.mx, Facultad de Ciencias, Universidad Nacional Aut\'onoma de M\'exico, M\'exico}}
\date{}

\begin{document}
\maketitle
\noindent Abstract: We study the second order invariants of a Lorentzian surface in $\R^{2,2},$ and the curvature hyperbolas associated to its second fundamental form. Besides the four natural invariants, new invariants appear in some degenerate situations. We then introduce the Gauss map of a Lorentzian surface and give an extrinsic proof of the vanishing of the total Gauss and normal curvatures of a compact Lorentzian surface. The Gauss map and the second order invariants are then used to study the asymptotic directions of a Lorentzian surface and discuss their causal character. We also consider the relation of the asymptotic lines with the mean directionally curved lines. We finally introduce and describe the quasi-umbilic surfaces, and the surfaces whose four classical invariants vanish identically.
\section*{Introduction}

Let $\R^{2,2}$ be the space $\R^4$ with the metric 
$$g=-dx_1^2+dx_2^2-dx_3^2+dx_4^2.$$ 

A surface $M\subset \R^{2,2}$ is said to be Lorentzian if the metric $g$ induces a Lorentzian metric, i.e. a metric of signature $(1,1),$ on $M:$ the tangent and the normal bundles $TM$ and $NM$ of a Lorentzian surface are equipped with Lorentzian fibre metrics. The second fundamental form at a point $p$ of a Lorentzian surface $M$ is a quadratic map $T_pM\rightarrow N_pM.$ The numerical invariants of the second fundamental form are second order invariants of the surface at $p,$ and locally determine the extrinsic geometry of the surface in $\R^{2,2}.$ The first purpose of the paper is to completely determine these invariants: additionally to the 4 natural invariants $|\H|^2,$ $K,$ $K_N$ and $\Delta$ which are the norm of the mean curvature vector, the Gauss curvature, the normal curvature and the resultant of the second fundamental form traducing the local convexity of the surface, new invariants appear in some degenerate cases.  A systematic study of the numerical invariants of a quadratic map $\R^{1,1}\rightarrow \R^{1,1}$ is necessary for this complete description. The second order invariants of surfaces and their geometric meaning have been extensively studied in different settings. In \cite{L} J. Little studied them in the case of a surface immersed in 4-dimensional Euclidian space. The second order invariants of a spacelike and a timelike surface in 4-dimensional Minkowski space were systematically studied in \cite{b_s_1} and \cite{b_s_3}. With the study of the quadratic maps between two Lorentzian planes, the present paper thus completes the description of the second order invariants of surfaces in 4-dimensional pseudo-Euclidian spaces. 

We then introduce the notion of curvature hyperbola associated to a quadratic map $\R^{1,1}\rightarrow \R^{1,1},$ which is analogous to the classical notion of curvature ellipse introduced in the Euclidian setting  \cite{L,W}. Its geometric properties may be naturally given in terms of the invariants of the quadratic map. When applied to the second fundamental form of a Lorentzian surface in $\R^{2,2},$ the curvature hyperbola gives a useful local representation of the surface. More cases appear than in the classical Euclidian case.

With these algebraic preliminaries at hand, we then study Lorentzian surfaces in $\R^{2,2}.$ We first introduce the Gauss map of an oriented Lorentzian surface. We show that the Gauss and the normal curvatures are obtained taking the pull-back by the Gauss map of the Lie bracket in $\Lambda^2\R^{2,2}$; as a consequence of this formula we obtain an extrinsic proof of the well-known fact that the total Gauss and normal curvatures vanish for a compact Lorentzian surface in $\R^{2,2}.$

We then use the preceding results  to introduce the notion of asymptotic directions of a Lorentzian surface in $\R^{2,2}$ and in Anti de Sitter space; we especially discuss the causal character of the asymptotic lines in terms of the invariants. Moreover, we relate these directions with the contact directions associated to the family of height functions on a Lorentzian surface $M$ in $\R^{2,2}$
 \cite{Ch-I}. We also introduce the mean directionally curved lines on a Lorentzian surface and specify their relation with the asymptotic lines.

We finally study the quasi-umbilic surfaces in $\R^{2,2},$ which are defined as the Lorentzian surfaces whose curvature hyperbolas degenerate at every point to a line with one point removed; alternatively, they are non-umbilic surfaces such that 
$$|\H|^2=K\hspace{1cm}\mbox{and}\hspace{1cm}K_N=\Delta=0$$
at every point. We then describe the Lorentzian surfaces in $\R^{2,2}$ whose classical invariants $|\H|^2,$ $K,$ $K_N$ and $\Delta$ vanish identically: they are surfaces in degenerate hyperplanes or flat umbilic or quasi-umbilic surfaces. In \cite{c}, J. Clelland  introduced and described the quasi-umbilic surfaces in 3-dimensional Minkowski space. The results of this last paper were then extended to the 4-dimensional Minkowski space in  \cite{b_s_3}; in the present paper, the results concerning the quasi-umbilic surfaces in $\R^{2,2}$ may also be considered as extending the main results of \cite{c}.
\\

The outline of the paper is as follows: we first study the quadratic maps from the Lorentz plane $\R^{1,1}$ into itself and their numerical invariants in Section \ref{section invariants}, and describe the curvature hyperbola associated to such a quadratic map in Section \ref{section curvature hyperbola}; we then study the Gauss map of a Lorentzian surface in Section \ref{section Gauss map}, and the asymptotic lines and the mean directionally curved lines of a Lorentzian surface in $\R^{2,2}$ and in Anti-de Sitter space in Section \ref{section asymptotic directions}. In Section \ref{section quasi umbilic}, we finally introduce the notion of quasi-umbilic surfaces and describe the surfaces which are umbilic or quasi-umbilic, and also the surfaces whose classical invariants vanish identically.

\section{Quadratics maps $\R^{1,1}\to \R^{1,1}$ and their numerical invariants}\label{section invariants}

Let $\R^{1,1}$ be the vector space $\R^2$ equipped with the Lorentzian metric 
$$\la\cdot,\cdot\ra:=-dx_1^2+dx_2^2.$$  
We will say that a non-zero vector $X$ belonging to $\R^{1,1}$ is spacelike (resp. timelike, or lightlike) if  its Lorentzian norm $\la X,X\ra$ is positive (resp. negative, or null).

We denote by $Q(\R^{1,1},\R^{1,1})$ the vector space of quadratics maps from $\R^{1,1}$ to $\R^{1,1}.$ We suppose that $\R^{1,1}$ is canonically oriented in space and in time: the canonical basis of $\R^{1,1}$ defines the orientation and a timelike vector in $\R^{1,1}$ is said to be future-directed if its first component in the canonical basis is positive. We consider the reduced (connected) group $SO(1,1)$  of Lorentzian direct isometries of $\R^{1,1}.$ 
This group acts on $Q(\R^{1,1},\R^{1,1})$ by composition (on the left and on the right)
\begin{align*}
SO(1,1)\times Q(\R^{1,1},\R^{1,1})\times SO(1,1) & \to Q(\R^{1,1},\R^{1,1})\\
(g_1,q,g_2) & \to g_1\circ q\circ g_2. 
\end{align*} 
In this section, we are interested in the description of the quotient set 
\[SO(1,1) \backslash Q(\R^{1,1},\R^{1,1})/SO(1,1);\]
specifically, we define numerical invariants on this quotient set which lead to a classification of the quadratic maps $\R^{1,1}\rightarrow\R^{1,1}$ up to the actions of $SO(1,1)$ (Theorem \ref{th classification}). The notion of \emph{quasi-umbilic} quadratic map will also emerge naturally.

\subsection{Forms associated to a quadratic map}

We fix $q\in Q(\R^{1,1},\R^{1,1}).$ If $\nu\in \R^{1,1},$ we denote by $S_{\nu}:\R^{1,1}\to \R^{1,1}$ the symmetric endomorphism associated to the real quadratic form $\la q,\nu \ra,$ i.e. such that
$$\la S_{\nu}(x),x\ra=\la q(x),\nu \ra$$
for all $x\in\R^{1,1}.$ For $\nu,\nu_1,\nu_2\in\R^{1,1}$ we define  \[L_q(\nu):=\frac{1}{2}\text{tr}(S_{\nu}),\hspace{0.1in}Q_q(\nu):=\det(S_{\nu})\hspace{0.1in}\text{and}\hspace{0.1in}A_q(\nu_1,\nu_2):=\frac{1}{2}[S_{\nu_1},S_{\nu_2}],\] 
where $[S_{\nu_1},S_{\nu_2}]$ denotes the morphism $S_{\nu_1}\circ S_{\nu_2}-S_{\nu_2}\circ S_{\nu_1};$ this morphism is skew-symmetric on $\R^{1,1},$ and thus identifies with the real number $\epsilon$ such that its matrix in the canonical basis of $\R^{1,1}$ is \[ \begin{pmatrix} 0 & \epsilon\\ \epsilon & 0
\end{pmatrix}.\] 

In the sequel, we will implicitly make this identification. We note that $L_q$ is a linear form, $Q_q$ is a quadratic form and $A_q$ is a bilinear skew-symmetric form on $\R^{1,1}.$ These forms are linked according to the following lemma:

\begin{lemma}\label{def Phi}
The quadratic form $\Phi_q:=L_q^2-Q_q$ satisfies the following identity: for all $\nu_1,\nu_2\in \R^{1,1},$
\begin{equation}\label{rela_formas}
\Phi_q(\nu_1)\Phi_q(\nu_2)=\tilde{\Phi}_q(\nu_1,\nu_2)^2-A_q(\nu_1,\nu_2)^2,
\end{equation} where $\tilde{\Phi}_q(\cdot,\cdot)$ denotes the symmetric bilinear form such that $\tilde{\Phi}_q(\nu,\nu)=\Phi_q(\nu)$ for all $\nu\in\R^{1,1}.$ In particular the signature of $\Phi_q$ is $(r,s)$ with $0\leq r,s\leq1.$
\end{lemma}
This lemma may be proved by a direct computation, using the representation of $S_{\nu_1}$ and $S_{\nu_2}$ by their matrices in the canonical basis of $\R^{1,1}.$ An alternative argument will also be given in Remark \ref{rmk alt arg} below.
\begin{remark}
The forms  $L_q,\Phi_q$ and $A_q$ are invariant by the right-action of $SO(1,1)$ on $q:$ for all $g\in SO(1,1)$ we have \[L_{q\circ g}=L_q,\hspace{0.1in}\Phi_{q\circ g}=\Phi_q\hspace{0.1in}\text{and}\hspace{0.1in}A_{q\circ g}=A_q.\]
They are thus also defined on the quotient set $Q(\R^{1,1},\R^{1,1})/SO(1,1).$
\end{remark} 

In the next section we will show the following: if $\Phi_q\neq 0,$ the forms $L_q,\Phi_q$ and $A_q$ determine $q$ up to the right-action of $SO(1,1)$; in the case $\Phi_q\equiv 0,$ $q$ is determined, up to the right-action of $SO(1,1),$ by the form $L_q$ together with some additional vector $\mu_q\in\R^{1,1}$ (Lemmas \ref{lem theta1} and \ref{des_vec} below).

\subsection{Reduction of a quadratic map}\label{section reduction quadratic map}
We denote by $\mathcal{S}$ the vector space of the traceless symmetric endomorphisms of $\R^{1,1}.$ $\mathcal{S}$ is naturally equipped with a metric tensor of signature $(1,1):$ if $u$ belongs to $\mathcal{S},$ we define its norm as \[|u|^2:=\frac{1}{2}\text{tr}(u^2).\] Expressing $u$ in the canonical basis of $\R^{1,1},$ we also easily get
\[|u|^2=-\det u.\]
Setting 
$$E_1:=\begin{pmatrix}1 & 0\\ 0 & -1 \end{pmatrix}\hspace{0.5cm}\text{and} \hspace{0.5cm}E_2:=\begin{pmatrix}0 & 1\\ -1 & 0 \end{pmatrix},$$ 
we have that $(E_1,E_2)$ is a Lorentzian basis for $\mathcal{S}$ such that $|E_1|^2=-|E_2|^2=1.$ Now, associated to a given quadratic map $q\in Q(\R^{1,1},\R^{1,1}),$ we consider the linear map 
\begin{eqnarray*}
f_q:\hspace{.5cm}\R^{1,1}& \to& \mathcal{S}\\ 
\nu &\mapsto& S_{\nu}^0:=S_{\nu}-L_q(\nu)I;
\end{eqnarray*}
for $\nu\in\R^{1,1},$ $f_q(\nu)$ is thus the traceless part $S_{\nu}^0$ of the symmetric operator $S_{\nu}$.
\begin{remark}\label{rmk alt arg}
It is not difficult to prove the following: for all $\nu_1,\nu_2\in\R^{1,1}$ we have 
\begin{equation}\label{pull_f} 
\tilde{\Phi}_q(\nu_1,\nu_2)=\la f_q(\nu_1),f_q(\nu_2)\ra \hspace{.5cm} \text{and} \hspace{.5cm} A_q(\nu_1,\nu_2)=\det{}_{(E_1,E_2)}(f_q(\nu_1),f_q(\nu_2)),
\end{equation}
where, if $s$ and $s'$ belong to $\mathcal{S},$ $\langle s,s'\rangle$ and $\det_{(E_1,E_2)}(s,s')$ stand respectively for the scalar product and for the determinant  in the basis $(E_1,E_2)$ of $s$ and $s'$ (considered as vectors of the Lorentzian plane $\mathcal{S}$). Formulas (\ref{pull_f}) and the Lagrange identity in the Lorentzian plane $(\mathcal{S},\la\cdot,\cdot\ra)$ give a direct proof of (\ref{rela_formas}).
\end{remark}

We recall the following convention concerning the orientation of $\R^{1,1}$: a basis $(e_1,e_2)$ of $\R^{1,1}$ is \textit{positively oriented} if it has the orientation of the canonical basis and if the vector $e_1$ is timelike and future-directed, i.e. is such that its first component in the canonical basis is positive (see the introduction of this section). 

\begin{remark}\label{form_cano}
If $u$ belongs to $\mathcal{S},$ $u\neq 0,$ its norm $|u|^2=-\det(u)$ determines its canonical form as follows: $u$ is diagonalizable if and only if $|u|^2>0,$ i.e. if and only if $u\in\mathcal{S}$ is spacelike; in that case, 
\[u=\pm \sqrt{|u|^2}E_1\] 
in some positively oriented and orthonormal basis of $\R^{1,1}.$ If $|u|^2<0$ ($u$ is timelike in $\mathcal{S}$), then 
\begin{equation}\label{reduction timelike}
u=\pm \sqrt{-|u|^2}E_2
\end{equation}
in some positively oriented and orthonormal basis of $\R^{1,1}.$ Finally, if $|u|^2=0,$ setting 
$$N_1:=\frac{1}{2}(E_1+E_2)\hspace{.5cm}\mbox{and}\hspace{.5cm}N_2:=\frac{1}{2}(E_2-E_1),$$ 
then
\begin{equation}\label{reduction lightlike}
u=\pm N_i,\hspace{1cm}i=1\ \mbox{or}\ 2
\end{equation}
in some positively oriented and orthonormal basis of $\R^{1,1}.$ 
\end{remark}

We now consider the reduction of a quadratic map $q\in Q(\R^{1,1},\R^{1,1}),$ and divide the discussion in three cases, according to the ranks of $f_q$ and $\Phi_q.$

\begin{enumerate}
\item $\bf rang(f_q)=2,$ or $\bf rang(f_q)=1$ with $\bf \Phi_q\neq 0.$ In that case, there is an orthonormal and positively oriented basis $(e_1,e_2)$ of $\R^{1,1}$ such that, in $(e_1,e_2),$ $S_{\nu},$ for all $\nu\in\R^{1,1},$ has the following canonical form:
\begin{center}
\begin{tabular}{|l|l|c|}\hline
Rank $f_q$ & Signature of $\Phi_q$ & Canonical form of $S_{\nu}$\\            
\hline \hline
2 & (1,1) &  $S_{\nu}=L_q(\nu)I\pm (\tilde{\Phi}_q(\nu_0,\nu)E_1+A_q(\nu_0,\nu)E_2)$ \\ \hline
1 & (1,0) & $S_{\nu}=L_q(\nu)I\pm \tilde{\Phi}_q(\nu_0,\nu)E_1$ \\ \hline
1 & (0,1) & $S_{\nu}=L_q(\nu)I\pm \tilde{\Phi}_q(\nu_0,\nu)E_2$\\ \hline
\end{tabular}
\end{center}
In the table, $\nu_0$ is some vector belonging to $\R^{1,1}.$ We only give brief indications of the proof, since similar results are proved in \cite{b_s_3}. In the first case, we consider $\nu_0$ such that $\Phi(\nu_0)=1;$ from the remark above, $S_{\nu_0}=L(\nu_0)I\pm E_1$ in some orthonormal and positively oriented basis $(e_1,e_2)$ of $\R^{1,1}$.  In $(e_1,e_2)$ and for an arbitrary $\nu\in\R^{1,1},$ $S_{\nu}$ may be a priori written
$$S_{\nu}=L(\nu)I\pm\left(a_{\nu}E_1+\b_{\nu}E_2\right)$$
for some $a_\nu,\b_\nu$ belonging to $\R.$ Straightforward computations using (\ref{pull_f}) then give $a_\nu=\tilde{\Phi}_q(\nu_0,\nu)$ and $b_\nu=A_q(\nu_0,\nu)$ and thus the required expression. The other cases may be proved similarly (taking $\nu_0$ such that $\Phi(\nu_0)=-1$ in the last case).

We then define 
\[Q_1(\R^{1,1},\R^{1,1}):=\{q\in Q(\R^{1,1},\R^{1,1}):\Phi_q\neq 0\}.\] Setting \[P_1=\{(L,\Phi,A):\Phi\hspace{.2cm}\text{is not zero, has a non-positive discriminant, and (\ref{rela_formas}) holds}\}
\] where $L,\Phi$ and $A$ are respectively linear, bilinear symmetric and skew-symmetric forms on $\R^{1,1},$ the following result holds:

\begin{lemma}\label{lem theta1}
The map $\Theta_1:Q_1(\R^{1,1},\R^{1,1})/SO(1,1)\longrightarrow P_1$ given by \[ [q]\longmapsto (L_{[q]},\Phi_{[q]},A_{[q]})\] is surjective and two-to-one.
\end{lemma}

We refer to \cite{b_s_3} for details, where a similar result is proved.

By the natural left-action of $SO(1,1)$ on $Q_1(\R^{1,1},\R^{1,1})/SO(1,1),$ the forms $L_{[q]},$ $\Phi_{[q]}$ transform as 
\[L_{g\circ [q]}=L_{[q]}\circ g^{-1},\hspace{1cm}\Phi_{g\circ [q]}=\Phi_{[q]}\circ g^{-1},\] whereas the form $A_{[q]}$ is invariant. Thus, if $SO(1,1)$ acts on $P_1$ by 
\begin{equation}\label{action P1}
g.(L,\Phi,A)=(L\circ g^{-1},\Phi\circ g^{-1},A),
\end{equation}
the map $\Theta_1$ is $SO(1,1)$-equivariant and thus induces a twofold map 
\begin{equation}\label{def theta1b}
\overline{\Theta}_1:SO(1,1)\backslash Q_1(\R^{1,1},\R^{1,1})/SO(1,1)\longrightarrow SO(1,1)\backslash P_1.
\end{equation} 
Since the formula (\ref{rela_formas}) permits the recovering of $A$ (up to sign) from $\Phi,$ the description of the quotient set $SO(1,1)\backslash Q_1(\R^{1,1},\R^{1,1})/SO(1,1)$ will be achieved with the simultaneous reduction of the forms $L_{[q]}$ and $\Phi_{[q]}.$ This is the aim of the first part of Section \ref{section sim red}.
 
\item $\bf rang(f_q)=1$ and $\bf \Phi_q=0.$ In that case, $f_q(\R^{1,1})$ is a line in $\mathcal{S},$ which is lightlike, and we have:
\begin{lemma}\label{vec_dist}There is a vector $\mu_q\in\R^{1,1},$ unit spacelike or timelike, or lightlike distinguished, and an orthonormal and positively oriented basis $(e_1,e_2)$ of $\R^{1,1}$ such that, for all $\nu\in\R^{1,1},$ the matrix of $S_{\nu}$ in $(e_1,e_2)$ is given by \begin{equation}\label{endo_vec_dist} S_{\nu}=L_q(\nu)I+\la\mu_q,\nu\ra N, \end{equation} where $N=N_1$ or $N_2$ (see Remark \ref{form_cano}). The vector $\mu_q$ and the basis $(e_1,e_2)$ are uniquely defined.
\end{lemma}
\begin{proof} In the canonical basis of $\R^{1,1},$ we have $S_{\nu}^0=\lambda_q(\nu)N,$ where $N=N_1$ or $N_2,$ and where $\lambda_q$ is a linear form on $\R^{1,1}.$ We define $\mu_q\in\R^{1,1}$ such that $\lambda_q(\nu)=\la \mu_q,\nu\ra$ for all $\nu\in\R^{1,1}.$ We now consider the basis of $\R^{1,1}$ obtained from the canonical basis by a Lorentzian rotation of angle $\psi.$ The matrix of $S_{\nu}^0$ in this basis is $S_{\nu}^0=\la e^{2\psi}\mu_q,\nu\ra N.$  Thus, there is a unique orthonormal and positively oriented basis of $\R^{1,1}$  such that in that basis $S_{\nu}^0=\la \mu_q,\nu\ra N$ with $|\mu_q|^2=+ 1,-1$ or $\mu_q=\frac{1}{2}(\pm 1,\pm1).$ 
\end{proof}

We now set
$$Q_2(\R^{1,1},\R^{1,1}):=\{q\in Q(\R^{1,1},\R^{1,1}):\Phi_q=0,f_q\neq0\}$$
and
$$P_2:=\R^{1,1}_*\times \mathcal{H}_0$$
where $\R^{1,1}_*$ stands for the set of linear forms on $\R^{1,1}$ and
$$\mathcal{H}_0:=\left\{\mu\in\R^{1,1}:\ |\mu|^2=\pm 1\hspace{0.1in}\text{or}\hspace{0.1in}\mu=\frac{1}{2}(\pm 1,\pm 1)\right\}.$$
\begin{lemma}\label{des_vec}
The map 
\begin{eqnarray*}
\Theta_2:\hspace{.5cm}Q_2(\R^{1,1},\R^{1,1})/SO(1,1)&\longrightarrow& P_2\\ 
{[q]}&\longmapsto& (L_{[q]},\mu_{[q]})
\end{eqnarray*}
is surjective and two-to-one.
\end{lemma}
\begin{proof} To each pair $(L,\mu)\in P_2$ correspond two classes in $Q_2(\R^{1,1},\R^{1,1})/SO(1,1),$ defining $S_{\nu}$ in the canonical basis of $\R^{1,1}$ by (\ref{endo_vec_dist}), where $N$ may be chosen to be $N_1$ or $N_2.$ \end{proof}

If $SO(1,1)$ acts on $P_2$ by \[g.(L,\mu)=(L\circ g^{-1},g.\mu),\] where $g.\mu=g(\mu)$ if $|\mu|^2=\pm 1,$ and $g.\mu=\mu$ if $|\mu|^2=0,$ the map $\Theta_2$ is $SO(1,1)$-equivariant and thus induces a twofold map \[\overline{\Theta}_2:SO(1,1)\backslash Q_2(\R^{1,1},\R^{1,1})/SO(1,1)\longrightarrow SO(1,1)\backslash P_2.\] Thus, the description of the quotient set $SO(1,1)\backslash Q_2(\R^{1,1},\R^{1,1})/SO(1,1)$ will be achieved with the simultaneous reduction of the form $L_{[q]}$ and the vector $\mu_{[q]}\in\mathcal{H}_0.$ This is the aim of the second part of Section \ref{section sim red}.

\item $\bf f_q=0.$ In that case, $S_{\nu}=L_q(\nu)I$ for all $\nu\in\R^{1,1}.$ We define \[Q_3(\R^{1,1},\R^{1,1}):=\{q\in Q(\R^{1,1},\R^{1,1}):\Phi_q=0,f_q=0\}.\] Setting $P_3:=\R^{1,1}_*,$ the map 
\begin{eqnarray*}
\overline{\Theta}_3:\hspace{.5cm}SO(1,1)\backslash Q_3(\R^{1,1},\R^{1,1})/SO(1,1)&\longrightarrow& SO(1,1)\backslash P_3\\
{[q]} &\longmapsto& [L_{[q]}]
\end{eqnarray*}
is bijective, where the action of $SO(1,1)$ on $P_3$ is given by $g.L=L\circ g^{-1}$.
\end{enumerate}
We finally define the notions of quasi-umbilic and umbilic quadratic maps, which correspond to the last two cases considered above:
\begin{definition}\label{def quasi-umbilic}
A quadratic map $q:\R^{1,1}\rightarrow\R^{1,1}$ is said to be \emph{quasi-umbilic} if 
$$rank(f_q)=1\hspace{.5cm}\mbox{and}\hspace{.5cm}\Phi_q=0;$$ 
this equivalently means that $f_q(\R^{1,1})$ is a lightlike line in $\mathcal{S}.$ A quadratic map $q:\R^{1,1}\rightarrow\R^{1,1}$ is said to be \emph{umbilic} if $f_q=0.$
\end{definition}
\subsection{Invariants on the quotient set}
In this section, we define invariants on the quotient set $SO(1,1) \backslash Q(\R^{1,1},\R^{1,1})/SO(1,1)$ associated to $L_{[q]},Q_{[q]},A_{[q]}$ and $\Phi_{[q]}.$
\begin{definition}\label{def invariants} 
Let $q:\R^{1,1}\rightarrow\R^{1,1}$ be a quadratic map, and $[q]\in Q(\R^{1,1},\R^{1,1})/SO(1,1)$ its class up to the right action of $SO(1,1).$ We consider 
\begin{enumerate}
\item the vector $\vec{H}\in\R^{1,1}$ such that, for all $\nu\in\R^{1,1},$ $L_{[q]}(\nu)=\la\vec{H},\nu\ra,$ and its norm \[|\vec{H}|^2:=\la\vec{H},\vec{H}\ra;\]
\item the two real numbers \[K:=tr\hspace{0.04in}Q_{[q]}\hspace{1cm} \text{and}\hspace{1cm}\Delta:=\det Q_{[q]},\] where $tr\hspace{0.04in}Q_{[q]}$ and $\det Q_{[q]}$ are the trace and the determinant of the symmetric endomorphism of $\R^{1,1}$ associated to $Q_{[q]}$ by the metric $\la\cdot,\cdot\ra$ on $\R^{1,1};$
\item the real number $K_N$ such that 
\[A_{[q]}=\frac{1}{2}K_N\ \omega_0,\] 
where $\omega_0$ is the determinant in the canonical basis of $\R^{1,1}$ (the canonical area form on $\R^{1,1}$).
\end{enumerate}
\end{definition}

The numbers $|\vec{H}|^2,K,K_N$ and $\Delta$ are kept invariant by the left-action of $SO(1,1)$ on $[q]\in Q(\R^{1,1},\R^{1,1})/SO(1,1)$ and thus define invariants on the quotient set $SO(1,1)\backslash Q(\R^{1,1},\R^{1,1})/SO(1,1).$

\begin{remark}
When the element of the quotient is given by the second fundamental form of a Lorentzian surface in $\R^{2,2}$ (see Section \ref{section Gauss map}), $\vec{H},$ $K$ and $K_N$ correspond to the mean curvature vector, the Gauss curvature and the normal curvature of the surface; the invariant $\Delta$ is similar to the invariant $\Delta$ introduced in \cite{L} for surfaces in $\R^4.$ This is naturally the motivation for these definitions.
\end{remark}

\begin{remark}Let $U_{\Phi}$ be the symmetric endomorphism on $\R^{1,1}$ associated to the quadratic form $\Phi_{[q]}.$ Denoting by $\text{tr}\hspace{0.04in}\Phi_{[q]}$ and $\det \Phi_{[q]}$ its trace and its determinant, we have
\begin{equation}\label{tr det phi}
\text{tr}\hspace{0.04in}\Phi_{[q]}=|\vec{H}|^2-K\hspace{1cm}\text{and}\hspace{1cm}\det \Phi_{[q]}=\frac{1}{4}K_N^2.
\end{equation}
These formulas may be proved by direct computations using the very definitions of $\Phi_{[q]}$ and the invariants; they will be useful below.
\end{remark}

\subsection{The last simultaneous reductions}\label{section sim red}
Accordingly to the previous sections, we have to consider two cases:
\\
\\{\bf 1. Case $\Phi_{[q]}\neq 0$.} In this case, $[q]\in Q(\R^{1,1},\R^{1,1})/SO(1,1)$ is determined by $\Phi_{[q]}$ and $L_{[q]}$ (Lemma \ref{lem theta1}); we thus reduce the operator $U_{\Phi},$ together with the mean curvature vector $\vec{H}:$

\begin{proposition}\label{diag} $U_{\Phi}$ is diagonalizable if and only if $U_{\Phi}^0=0$ or $$(|\vec{H}|^2-K)^2-K_N^2>0.$$ In that last case, there is a unique orthonormal and positively oriented basis $(u_1,u_2)$ of $\R^{1,1}$ such that the matrix of $U_{\Phi}$ in $(u_1,u_2)$ is 
\begin{equation}\label{diag2} \begin{pmatrix}\a & 0 \\0 & \b\end{pmatrix} 
\end{equation} 
where
\begin{equation}\label{diagigual1}
\a:=\frac{|\vec{H}|^2-K\ \pm\sqrt{(|\vec{H}|^2-K)^2-K_N^2}}{2}
\end{equation}
and
\begin{equation}\label{diagigual2}
\b:=\frac{|\vec{H}|^2-K\ \mp\sqrt{(|\vec{H}|^2-K)^2-K_N^2}}{2}.
\end{equation}
Moreover, defining $\alpha,\beta\in\R$ such that $\vec{H}=\alpha u_1+\beta u_2,$ we have
\begin{equation}\label{inva1} 
\alpha^2=\frac{1}{\b-\a} \left(\a |\vec{H}|^2+\Delta-\frac{1}{4}K_N^2\right)
\end{equation}
and
\begin{equation}\label{inva2} 
\beta^2=\frac{1}{\b-\a}\left( \b |\vec{H}|^2+\Delta-\frac{1}{4}K_N^2\right).
\end{equation}
\end{proposition}
\begin{proof}
The first part of the proposition follows from the fact that $U_{\Phi}$ is diagonalizable if and only if $U_{\Phi}^0=0$ or $1/4(\mbox{tr}\ U_{\Phi})^2>\det U_{\Phi},$ together with (\ref{tr det phi}) ($U_{\Phi}^0$ is spacelike in $\mathcal{S},$ see Remark \ref{form_cano}). For the second part of the statement, we consider the quadratic form $Q=L^2-\Phi$ and its associated symmetric operator $U_{Q}:\R^{1,1}\rightarrow\R^{1,1};$ its matrix in $(u_1,u_2)$ is
 $$U_{Q}=\left(\begin{array}{cc}-\alpha^2-a&\alpha\beta\\-\alpha\beta&\beta^2-b\end{array}\right).$$
The formulas $\mbox{tr}\ U_{Q}=K,$ $\det U_{Q}=\Delta$ (Definition \ref{def invariants}) and (\ref{tr det phi}) easily give (\ref{inva1}) and (\ref{inva2}).
\end{proof}

\begin{proposition}\label{nodia}
$U_{\Phi}$ is not diagonalizable if and only if 
\begin{equation}\label{cond uphi nodiag}
U^0_{\Phi}\neq 0\hspace{.5cm}\mbox{ and }\hspace{.5cm}(|\vec{H}|^2-K)^2-K_N^2\leq 0,
\end{equation}
and we have:
\\\textbf{1.} if $(|\vec{H}|^2-K)^2-K_N^2<0,$ there is a unique orthonormal and positively oriented basis $(u_1,u_2)$ of $\R^{1,1}$ such that the matrix of $U_{\Phi}$ in $(u_1,u_2)$ is
\begin{equation}\label{nodiag1}
\frac{|\vec{H}|^2-K}{2}\begin{pmatrix}1 & 0\\0 & 1\end{pmatrix}\pm \frac{\sqrt{K_N^2-(|\vec{H}|^2-K)^2}}{2}\begin{pmatrix}0 & 1\\-1 & 0\end{pmatrix}. 
\end{equation} 
Writing $\vec{H}=\alpha u_1+\beta u_2,$ we have
\begin{equation}\label{reltiemp}
\alpha^2=\frac{1}{2}\left( -|\vec{H}|^2+ \sqrt{|\vec{H}|^4+4u^2} \right),\hspace{1cm} \beta^2=\frac{1}{2}\left(|\vec{H}|^2+ \sqrt{|\vec{H}|^4+4u^2} \right),
\end{equation}
where 
$$u=\frac{1}{\sqrt{K_N^2-(|\vec{H}|^2-K)^2}}\left(-\Delta+\frac{1}{4}K_N^2-\frac{1}{2}|\vec{H}|^2(|\vec{H}|^2-K)\right).$$
\textbf{2.} if $(|\vec{H}|^2-K)^2-K_N^2=0,$ there is a unique orthonormal and positively oriented basis $(u_1,u_2)$ of $\R^{1,1}$ such that the matrix of $U_{\Phi}$ in $(u_1,u_2)$ is
\begin{equation}\label{uphi epsilons}
\frac{|\vec{H}|^2-K}{2}\begin{pmatrix}1 & 0 \\0 & 1\end{pmatrix}+ \begin{pmatrix}\varepsilon_1 & \varepsilon_2 \\-\varepsilon_2 & -\varepsilon_1\end{pmatrix}\hspace{1cm}
\end{equation}
where $\varepsilon_1=\pm 1,$ $\varepsilon_2=\pm 1.$ Writing $\vec{H}=\alpha u_1+\beta u_2$, we have
\begin{equation}\label{alpha beta timelike}
\alpha^2=\frac{\varepsilon_1}{4} \dfrac{\left(|\vec{H}|^2-\varepsilon_1v\right)^2}{v},\hspace{1cm} \beta^2=\frac{\varepsilon_1}{4} \dfrac{\left(|\vec{H}|^2+\varepsilon_1v\right)^2}{v},
\end{equation}
if $\displaystyle{v:=-\Delta+\frac{K_N^2}{4}-\frac{1}{2}|\vec{H}|^2(|\vec{H}|^2-K)}$ is not 0. Moreover, $v=0$ if and only if $|\H|^2=0;$ in that case,
\begin{equation}\label{invariants phi H lightlike}
\Delta=\frac{1}{4}K_N^2=\frac{1}{4}K^2,\hspace{1cm}|\H|^2=0
\end{equation}
and
\begin{equation}\label{def invariants Phi neq 0}
\H=\alpha u_1+\beta u_2,\hspace{1cm}\mbox{with}\hspace{.3cm} \alpha=\pm \beta
\end{equation}
defines new invariants $\alpha,\beta.$
\end{proposition}
\begin{proof}
\textbf{\textit{ 1.}} In that case $U_{\Phi}^0$ is timelike in $\mathcal{S},$ and its reduction is given by (\ref{reduction timelike}) in Remark \ref{form_cano}, which proves (\ref{nodiag1}). Formulas (\ref{reltiemp}) may then be proved as formulas (\ref{inva1}) and (\ref{inva2}) in Proposition \ref{diag} above.
\\\textbf{\textit{2.}} Here $U_{\Phi}^0$ is lightlike in $\mathcal{S},$ and its reduction is given by (\ref{reduction lightlike}) in Remark \ref{form_cano}, which proves (\ref{uphi epsilons}). We also get formulas (\ref{alpha beta timelike}) as in Proposition \ref{diag} above. Further, computing $\Delta=\det U_{Q}$ as in the proof of Proposition \ref{diag}, with $U_{\Phi}$ given here by (\ref{uphi epsilons}), we may easily get
$$v=\varepsilon_1(\alpha^2+\beta^2)+2\varepsilon_2\alpha\beta.$$
Thus $v=0$ if and only if $\alpha=\pm\beta,$ i.e. $|\H|^2=0;$ formulas (\ref{invariants phi H lightlike}) then easily follow.  
\end{proof}
\noindent {\bf 2. Case $\Phi_{[q]}=0$.} In this case, and if $f_q\neq 0,$ $[q]\in Q(\R^{1,1},\R^{1,1})/SO(1,1)$ is determined by the form $L_{[q]}$ together with the vector $\mu_{[q]}$ (Lemma \ref{vec_dist}), and we need to simultaneously reduce $L_{[q]}$ and $\mu_{[q]}.$ We recall that $\mu_{[q]}$ is normalized so that $|\mu_{[q]}|^2=\pm1,$ or $\mu_{[q]}=\frac{1}{2}(\pm1,\pm1),$ and we define the vector 
\begin{equation}
\mu_{[q]}^*:=\begin{cases}\mu_{[q]}^\perp,\hspace{0.2in}\text{if}\hspace{0.2in}|\mu_{[q]}|^2=\pm1 \\ \mu_{[q]}',\hspace{0.2in}\text{if}\hspace{0.2in}\mu_{[q]}=\frac{1}{2}(\pm1,\pm1)
\end{cases} 
\end{equation}where $\mu_{[q]}^\perp$ denotes the reflection of $\mu_{[q]}$ with respect to the principal diagonal of $\R^{1,1}$ in the first case, and $\mu_{[q]}'$ the unique lightlike vector such that $\la\mu_{[q]},\mu_{[q]}'\ra=\frac{1}{2}$ in the second case. Now $(\mu_{[q]},\mu^*_{[q]})$ is a basis of $\R^{1,1},$ and we define $\alpha$ and $\beta$ such that
\begin{equation}\label{nuev_inva}
\vec{H}=\alpha \mu_{[q]}+\beta \mu_{[q]}^*.
\end{equation}  
The numbers $\alpha$ and $\beta$ are new invariants. We will give an interpretation of these invariants in Section \ref{section Gauss map} below.

\subsection{The classification}
We gather the results obtained in the previous sections and give the classification of the quadratic maps $\R^{1,1}\rightarrow\R^{1,1}$ in terms of their numerical invariants. For sake of simplicity, we will say that a set of invariants \emph{essentially} determines a class in $SO(1,1) \backslash Q(\R^{1,1},\R^{1,1})/SO(1,1)$ if it completely determines a \emph{finite number} of classes (corresponding to choices of signs in the formulas given in the previous sections). 
\begin{theorem}\label{th classification}
The class $[q]\in SO(1,1) \backslash Q(\R^{1,1},\R^{1,1})/SO(1,1)$ is determined by its invariants in the following way:
	\begin{enumerate}
	\item If $\Phi\neq 0$ then the following holds:
		\begin{enumerate}
	\item if $(|\vec{H}|^2-K)^2-K_N^2\neq 0,$ the invariants $K,K_N,|\vec{H}|^2,\Delta$ essentially determine $[q];$
	\item if $(|\vec{H}|^2-K)^2-K_N^2=0,$ then we have:
	\begin{enumerate}
	\item  if $|\vec{H}|^2\neq 0,$ the invariants $K,|\vec{H}|^2,\Delta$ essentially determine $[q];$
	\item  if $|\vec{H}|^2=0,$ the invariant $K$ together with the new invariants $\alpha,\beta$ defined in (\ref{def invariants Phi neq 0}) essentially determine $[q].$
	\end{enumerate}
\end{enumerate}
\item If $\Phi= 0$ then $K_N=0,$ $|\vec{H}|^2-K=0$ and $\Delta=0,$ and we have the following:
			\begin{enumerate}
				\item if $f_q\neq 0,$ the invariants $\alpha$ and $\beta$ defined in (\ref{nuev_inva}) essentially determine $[q]$ ($q$ is \emph{quasi-umbilic});
				\item if $f_q=0,$ then $|\vec{H}|^2$ determine $[q]$ ($q$ is \emph{umbilic}).
				\end{enumerate}
	\end{enumerate}
\end{theorem}
\begin{proof} 
We only consider the case $(|\vec{H}|^2-K)^2-K_N^2>0$ since the proofs in the other cases are very similar. Recall first the definition of $\overline{\Theta}_1$ in (\ref{def theta1b}), Section \ref{section reduction quadratic map}.  By Proposition \ref{diag}, $\overline{\Theta}_1([q])$ is the class of $(L,\Phi,A)\in P_1$ where the forms $L,\Phi$ and $A$ are defined \emph{in the canonical basis} $(u_1,u_2)$ of $\R^{1,1}$ by
$$L=(\alpha,\beta),\hspace{1cm}\Phi=\left(\begin{array}{cc} a&0\\0&b\end{array}\right)\hspace{.5cm}\mbox{and}\hspace{.5cm}A=\frac{1}{2}K_N\left(\begin{array}{cc}0&1\\-1&0\end{array}\right),$$
with $a,b,\alpha$ and $\beta$ satisfying (\ref{diagigual1})-(\ref{inva2}) (more precisely, recalling (\ref{action P1}), if $g\in SO(1,1)$ is such that $g(u_1) = \tilde{u}_1,$ $g(u_2) = \tilde{u}_2,$ where $(\tilde{u}_1, \tilde{u}_2)$ is the basis given by Proposition \ref{diag}, we have $g.(L,\Phi, A) = (L_q, \Phi_q, A_q)$). Since we can choose a sign in the definitions (\ref{diagigual1}) and (\ref{diagigual2}) of $a$ and $b,$ and since $\alpha$ and $\beta$ are determined up to sign by (\ref{inva1}) and (\ref{inva2}), sixteen classes correspond to the given set of invariants (two classes correspond to each one of the eight possible choices for $a,b, \alpha$ and $\beta$ since the map $\overline{\Theta}_1$ in (\ref{def theta1b}) is two-to-one). 
\end{proof}

\section{The curvature hyperbola of $q:\R^{1,1}\rightarrow\R^{1,1}$}\label{section curvature hyperbola}

In this section we describe the geometric properties of the curvature hyperbola associated to a quadratic map in terms of its invariants. \textit{The curvature hyperbola} $\mathcal{H}$ associated to $q:\R^{1,1}\to \R^{1,1}$ is defined as the subset of $\R^{1,1}$ \[\mathcal{H}:=\left\lbrace  \frac{q(v)}{|v|^2}:\ v\in \mathbb{R}^{1,1}, |v|^2=\pm 1\right\rbrace ;\] 
this is the natural analog of the curvature ellipse associated to a quadratic map $\R^2\rightarrow\R^2,$ where $\R^2$ is the Euclidian plane. Denoting by $\mathcal{O}$ the origin of $\R^{1,1},$ the center of $\mathcal{H}$ is the point $\mathcal{C}$ such that $\stackrel{\longrightarrow}{\mathcal{O}\mathcal{C}}=\vec{H}$ ($\vec{H}$ is the mean curvature vector of $q,$ see Definition \ref{def invariants}). We will say that a point $\mathcal{P}$ of the hyperbola is spacelike (resp. timelike) if the vector $\stackrel{\longrightarrow}{\mathcal{C}\mathcal{P}}$ is a spacelike (resp. timelike) vector of $\R^{1,1}.$ Generically, the curvature hyperbola is given as follows:
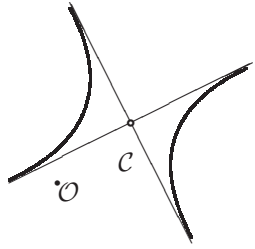
\begin{center}
\setlength{\unitlength}{0.8mm}
\begin{picture}(50,50)
\put(5,15){\line(2,1){40}}
\put(35,5){\line(-1,2){20}}
\thicklines
\qbezier(35,6)(25,25)(44,34)
\qbezier(5,15.5)(25,25)(15,44)
\put(25,25){\color{white}\circle*{1}}
\put(25,25){\circle{1}}
\put(23,17){$\mathcal{C}$}
\put(13,12){$\mathcal{O}$}
\put(13,15){\circle*{1}}
\end{picture}
\end{center}
We have the following descriptions of the hyperbola:
\begin{proposition}\label{hiperbola_no_degenerada}If $K_N\neq 0,$ the curvature hyperbola $\mathcal{H}$ is not degenerate, and the following holds:
\\
\\1- if $U_{\Phi}^0=0,$ the curvature hyperbola is 
$$\mathcal{H}=\left\{\H+\nu:\ |\nu|^2=\frac{|\H|^2-K}{2}\right\};$$ 
its asymptotes are two null lines in $\R^{1,1};$ 
\\
\\2- if $U_{\Phi}^0\neq 0$ and $U_{\Phi}$ is diagonalizable with eigenvalues $\a$ and $\b$ given by Proposition \ref{diag}, the axes of $\mathcal{H}$ are directed by the eigenvectors $u_1$ and $u_2$ of $U_{\Phi},$  and,  its equation in $(\vec{H},u_1,u_2)$ is					
$$\frac{\nu_2^2}{\b}-\frac{\nu_1^2}{\a}=1;$$
if $a>b$ (resp. $a<b$), its asymptotes are timelike (resp. spacelike) lines, and moreover, the hyperbola contains timelike and spacelike points (resp. contains only spacelike points) if $\a,\b>0$, and contains only timelike points (resp. contains timelike and spacelike points) if $\a,\b<0$;   
\\
\\3- if $U_{\Phi}$ is not diagonalizable, we have two cases which correspond to the cases in Proposition \ref{nodia}:
\\
\\a- if $U_{\Phi}^0$ is timelike (in $\mathcal{S}$), the equation of $\mathcal{H}$ is
			$$-\frac{2(|\vec{H}|^2-K)}{K_N^2}\nu_1^2\mp 4\frac{\sqrt{K_N^2-(|\vec{H}|^2-K)^2}}{K_N^2} \nu_1\nu_2+\frac{2(|\vec{H}|^2-K)}{K_N^2}\nu_2^2=1;$$
			one of the asymptotes is timelike and the other one is spacelike, and the hyperbola contains timelike and spacelike points;
\\
\\b- if $U_{\Phi}^0$ is lightlike (in $\mathcal{S}$), the equation of $\mathcal{H}$ is 
			$$\frac{1-\a}{\a^2}\nu_1^2+\frac{2\varepsilon}{\a^2}\nu_1\nu_2+\frac{1+\a}{\a^2}\nu_2^2=1\hspace{.5cm}\mbox{or}\hspace{.5cm}-\frac{1+\a}{\a^2}\nu_1^2+\frac{2\varepsilon}{\a^2}\nu_1\nu_2-\frac{1-\a}{\a^2}\nu_2^2=1$$ 
			where $\a=\frac{|\vec{H}|^2-K}{2}$ and $\varepsilon=\pm 1.$ If $\mathcal{H}$ is given by the first equation (resp. the second equation), it has a lightlike asymptote, which is the line $\nu_2=-\varepsilon\nu_1$ (resp. the line $\nu_2=\varepsilon\nu_1$); its other asymptote is timelike (resp. spacelike) if $\a>0,$ and is spacelike (resp. timelike) if $\a<0;$ moreover, $\mathcal{H}$ contains timelike and spacelike points (resp. only spacelike points, or only timelike points).
\end{proposition}

We omit the proof, which is quite long and elementary.

\begin{remark} If $K_N\neq 0,$ the function $U_{\Phi}$ is invertible, and we may define $\Phi^*(\nu):=\la\nu,U_{\Phi}^{-1}\nu\ra.$ It turns out that the function $\Phi^*:\R^{1,1}\to \R$ then furnishes an intrinsic equation of the curvature hyperbola: for all $\nu\in\R^{1,1},$ \[\vec{H}+\nu\ \in\mathcal{H}\hspace{0.2in}\text{if and only if}\hspace{0.2in}\Phi^*(\nu)=1.\]
This gives an efficient device to write down the equation of the curvature hyperbola in specific cases, since $U_{\Phi}$ may be easily written in terms of the second fundamental form.
\end{remark}

We also describe the curvature hyperbola in the degenerate case ($K_N=0$). Here again, for sake of brevity we omit the proofs.

\begin{proposition}\label{imagen_fq_1} If $K_N=0$ and $\Phi_q\neq 0,$ we have two possibilities: 
\\
\\1- the image of $f_q$ is a spacelike line; in this case the hyperbola degenerates to the union of two half-lines: 
\\
\\(a) if $|\vec{H}|^2-K\neq 0,$ the hyperbola is 
 		\[ \left\lbrace \vec{H}\pm\lambda\sqrt{|\vec{H}|^2-K}_{}u_2,\ 1\leq\lambda<+\infty\right\rbrace\hspace{.3cm}\text{or}\hspace{.3cm}\left\lbrace \vec{H}\pm\lambda\sqrt{K-|\vec{H}|^2}_{}u_1,\ 1\leq\lambda<+\infty \right\rbrace,\]
		depending on the sign of $|\vec{H}|^2-K;$ 
\\
\\(b) if $|\vec{H}|^2-K=0,$ the hyperbola is 
 		\[ \left\lbrace  \vec{H}\pm\lambda(u_1+u_2),\ 1\leq\lambda<+\infty \right\rbrace \hspace{.3cm}\text{or} \hspace{.3cm} \left\lbrace \vec{H}\pm\lambda(u_1-u_2),\ 1\leq\lambda<+\infty \right\rbrace;\]
		this occurs when $U_{\Phi}$ is given by (\ref{uphi epsilons}) with $\varepsilon_1=-1;$
\\			
\\2- the image of $f_q$ is a timelike line; in that case the hyperbola degenerates to a straight line:
\\
\\(a) if $|\vec{H}|^2-K\neq 0,$ the hyperbola is 
$$\H+\R u_1\hspace{1cm}\mbox{or}\hspace{1cm}\H+\R u_2,$$
where the first case occurs if $|\vec{H}|^2-K>0$ and the second case if $|\vec{H}|^2-K<0;$
\\
\\(b) if $|\vec{H}|^2-K=0,$ the hyperbola is 
$$\H+\R (u_1-u_2)\hspace{1cm}\mbox{or}\hspace{1cm}\H+\R (u_1+u_2);$$
this occurs when $U_{\Phi}$ is given by (\ref{uphi epsilons}) with $\varepsilon_1=1.$
\\
\\For both cases 1 and 2, the case $(a)$ corresponds to $U_{\Phi}$ diagonalizable and the case $(b)$ to $U_{\Phi}$ non diagonalizable, and the basis $(u_1,u_2)$ is given by Proposition \ref{diag} and Proposition \ref{nodia} respectively. We moreover note that $\Delta\geq 0$ in the case 1, and that  $\Delta\leq 0$ in the case 2. 
\end{proposition}
\begin{proposition}\label{imagen_fq_2}
If  $K_N=0$ and $\Phi_q=0,$ we consider two cases: 
\begin{enumerate}
\item $f_q\neq0;$ in that case the hyperbola degenerates to a straight line with one point removed
$$\left\lbrace \vec{H}+\lambda \mu_{q},\ \lambda\in\R\backslash\{0\} \right\rbrace,$$ 
where $\mu_{q}$ is the distinguished vector defined in Lemma \ref{vec_dist}; in that case $\Delta=0,$ and $q$ is quasi-umbilic;
\item $f_q=0;$ the hyperbola then degenerates to the end point of the vector $\vec{H};$ $q$ is umbilic.
\end{enumerate}
\end{proposition}
In the figure below, the hyperbolas $(a)$ and $(b)$ correspond to the first and to the second case in Proposition \ref{imagen_fq_1} respectively, and the hyperbola $(c)$ to the first case in Proposition \ref{imagen_fq_2}.
\vspace{.5cm}

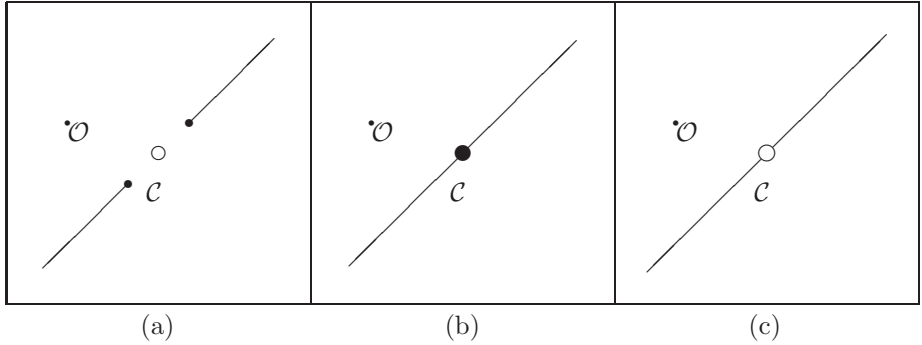
\begin{center}
\setlength{\unitlength}{0.8mm}
\begin{picture}(150,50)
\thinlines
\put(0,0){\line(1,0){150}}
\put(0,50){\line(1,0){150}}
\put(0,0){\line(0,1){50}}
\put(50,0){\line(0,1){50}}
\put(100,0){\line(0,1){50}}
\put(150,0){\line(0,1){50}}
\put(10,27){$\mathcal{O}$}
\put(10,30){\circle*{1}}

\put(20,20){\line(-1,-1){14}}
\put(30,30){\line(1,1){14}}
\put(25,25){\circle{2.3}}
\put(30,30){\circle*{1.5}}
\put(20,20){\circle*{1.5}}
\put(23,17){$\mathcal{C}$}
\put(60,27){$\mathcal{O}$}
\put(60,30){\circle*{1}}
\put(22,-5){(a)}

\put(75,25){\line(-1,-1){18.7}}
\put(75,25){\line(1,1){18.7}}
\put(75,25){\circle*{2.5}}
\put(73,17){$\mathcal{C}$}
\put(110,27){$\mathcal{O}$}
\put(110,30){\circle*{1}}
\put(72,-5){(b)}

\put(124,24){\line(-1,-1){18.7}}
\put(126,26){\line(1,1){18.7}}
\put(125,25){\color{white}\circle*{1}}
\put(125,25){\circle{2.5}}
\put(123,17){$\mathcal{C}$}
\put(122,-5){(c)}
\end{picture}
\end{center}
\vspace{.5cm}

\section{The Gauss map of a Lorentzian surface in $\R^{2,2}$}\label{section Gauss map}

Let $M$ be a Lorentzian surface immersed in $\R^{2,2}$. We will assume that $M$ is oriented in space and in time: the tangent and the normal bundles $TM$ and $NM$ are oriented, and for all $p\in M,$ a component of  $\{X\in T_pM,\ g(X,X)<0\}$ and a component of $\{X\in N_pM,\ g(X,X)<0\}$ are distinguished; a vector (tangent or normal to $M$) belonging to such a component will be called \emph{future-directed}. We will moreover adopt the following convention:  a basis $(u,v)$ of $T_pM$ or $N_pM$ will be said positively oriented (in space and in time) if it has the orientation of $T_pM$ or $N_pM$ and if $g(u,u)<0$ and $g(v,v)>0$ with $u$ future-directed. The second fundamental form $II:T_pM \rightarrow N_pM$ at each point $p \in M$ is a quadratic map between two (oriented) Lorentzian planes: such a quadratic map naturally defines an element of $SO(1,1)\backslash Q(\R^{1,1},\R^{1,1})/SO(1,1),$ given by its representation in positively oriented and orthonormal frames of $T_pM$ and $N_pM$; the numerical invariants and the curvature hyperbola introduced in the previous sections are thus naturally attached to the second fundamental form $II.$

Let us consider $\Lambda^2\R^{2,2},$ the vector space of bivectors of $\R^{2,2},$ endowed with its natural metric $\langle.,.\rangle$, which has signature $(2,4)$. The Grassmannian of the oriented Lorentzian 2-planes in $\R^{2,2}$ identifies with the submanifold of unit and simple bivectors
$$\mathcal{Q}=\{\eta\in\Lambda^2\R^{2,2}:\ \langle \eta,\eta\rangle=-1,\ \eta\wedge\eta=0\},$$
and the oriented Gauss map with the map
$$G:\hspace{.5cm} M\rightarrow \mathcal{Q},\hspace{.5cm} p\mapsto G(p)=u_1\wedge u_2,$$
where $(u_1,u_2)$ is a positively oriented and orthonormal basis of $T_pM$ (we recall that $u_1$ is timelike and $u_2$ is spacelike). We also consider the Lie bracket
$$[.,.]:\hspace{.5cm}\Lambda^2\R^{2,2}\times\Lambda^2\R^{2,2}\rightarrow\Lambda^2\R^{2,2}.$$
Its restriction to the submanifold $\mathcal{Q}$ is a 2-form with values in $\Lambda^2\R^{2,2}.$ It appears that its pull-back by the Gauss map gives the Gauss and the normal curvatures of the surface:
\begin{proposition}
If $\nabla$ denotes the Levi-Civita connection on $TM$ and $\nabla'$ the normal connection on $NM,$ we have
\begin{equation}\label{pull back general}
G^*[.,.]=R^{\nabla\oplus\nabla'},
\end{equation}
where $R^{\nabla\oplus\nabla'}$ is the curvature tensor of the connection $\nabla\oplus\nabla'$ on $TM\oplus NM,$ considered as a 2-form on $M$ with values in $\Lambda^2TM\oplus\Lambda^2NM\subset M\times\Lambda^2\R^{2,2}.$ 
\end{proposition}
We refer to \cite{AB} for a much more general result, in the Riemannian setting, where the bracket $[.,.]$ is interpreted as the curvature tensor of the tautological bundles on the Grassmannian. Although such an interpretation should be also possible here (and explain the result), we give a more direct proof. 
\begin{proof}
We assume that $(e_1,e_2)$ is a local frame of $TM$ in a neighborhood $\mathcal{U}$ of $p\in M$ such that 
$$|e_1|^2=-1,\ |e_2|^2=1\ \mbox{on}\ \mathcal{U}\hspace{.5cm}\mbox{ and}\hspace{.5cm}\nabla e_1=\nabla e_2=0\ \mbox{at}\ p,$$
and, since $G=e_1\wedge e_2,$ we readily get
$$dG(e_1)=II(e_1,e_1)\wedge e_2+e_1\wedge II(e_2,e_1)$$
and 
$$dG(e_2)=II(e_1,e_2)\wedge e_2+e_1\wedge II(e_2,e_2).$$
For the computation, it is convenient to consider $\Lambda^2\R^{2,2}$ as a subset of the Clifford algebra $Cl(2,2):$ the bracket $[.,.]$ is then simply given by
$$[\eta,\eta']=\frac{1}{2}\left(\eta'\cdot\eta-\eta\cdot\eta'\right)$$
for all $\eta,\eta'\in\Lambda^2\R^{2,2},$ where the dot "$\cdot$" stands for the Clifford product; see \cite{Fr} for the basic properties of the Clifford algebras. We then compute
$$[dG(e_1),dG(e_2)]=\frac{1}{2}\left(dG(e_2)\cdot dG(e_1)-dG(e_1)\cdot dG(e_2)\right)$$
with
$$dG(e_1)=II(e_1,e_1)\cdot e_2+e_1\cdot II(e_2,e_1)$$
and 
$$dG(e_2)=II(e_1,e_2)\cdot e_2+e_1\cdot II(e_2,e_2),$$
and easily get
$$[dG(e_1),dG(e_2)]=A+B$$
with
\begin{eqnarray*}
A&=&\left(-\langle II(e_1,e_1),II(e_2,e_2)\rangle+|II(e_1,e_2)|^2\right)e_1\cdot e_2\\
&=&K\ e_1\cdot e_2
\end{eqnarray*}
and
\begin{eqnarray*}
B&=&\frac{1}{2}\left\{-II(e_1,e_1)\cdot II(e_1,e_2)+II(e_2,e_1)\cdot II(e_2,e_2)\right.\\
&&\left.+II(e_1,e_2)\cdot II(e_1,e_1)-II(e_2,e_2)\cdot II(e_2,e_1)\right\}\\
&=& K_N\ e_3\cdot e_4
\end{eqnarray*}
where $(e_3,e_4)$ is a positively oriented and orthonormal frame of $N_{p}M$ ($|e_3|^2=-|e_4|^2=-1$). To derive these formulas we use that
$$e_1\cdot e_1=e_3\cdot e_3=1\hspace{1cm}\mbox{and}\hspace{1cm}e_2\cdot e_2=e_4\cdot e_4=-1,$$
together with the formulas 
$$K=xy-z^2-uv+w^2\hspace{1cm}\mbox{and}\hspace{1cm}K_N=-w(x+y)+z(u+v)$$
if the second fundamental form is given by
$$II=\left(\begin{array}{cc}x& z\\z&y\end{array}\right)e_3+\left(\begin{array}{cc}u& w\\w&v\end{array}\right)e_4$$
in $(e_1,e_2).$ See \cite{BLR} for details, where a similar computation is carried out. Thus
\begin{equation}\label{pull back K K_N}
G^*[.,.]=\left(K\ e_1\wedge e_2+K_N\ e_3\wedge e_4\right)\omega_M,
\end{equation}
which is equivalent to (\ref{pull back general}).
\end{proof}
\begin{corollary}
Let us consider the 2-forms $\omega_T$ and $\omega_N$ defined on $\mathcal{Q}$ by
\begin{equation}
{\omega_T}_p(\eta,\eta'):=-\langle p, [\eta,\eta']\rangle\hspace{.5cm}\mbox{and}
\hspace{.5cm}{\omega_N}_p(\eta,\eta'):=-\langle *p, [\eta,\eta']\rangle
\end{equation}
for all $p\in\mathcal{Q},$ $\eta,\eta'\in T_p\mathcal{Q}.$ Then
\begin{equation}\label{pull back omega T N}
G^*\omega_T=K\ \omega_M
\hspace{.5cm}\mbox{and}\hspace{.5cm}G^*\omega_N=K_N\ \omega_M,
\end{equation}
where $\omega_M$ is the area form of $M.$
\end{corollary}
In the statement of the corollary and below, $\langle.,.\rangle$ denotes the natural scalar product on $\Lambda^2\R^{2,2}$ and $*:\Lambda^2\R^{2,2}\rightarrow\Lambda^2\R^{2,2}$ is the Hodge operator, i.e. the symmetric operator of $\Lambda^2\R^{2,2}$ such that
$$\eta\wedge\eta'=\langle\eta,*\eta'\rangle\ e_1\wedge e_2\wedge e_3\wedge e_4$$
for all $\eta,\eta',$ where $e_1\wedge e_2\wedge e_3\wedge e_4$ is the canonical volume element.
\begin{proof}
By definition, we have
$$[.,.]_p=\omega_T\ p+\omega_N\ (*p)$$
for all $p\in\mathcal{Q},$ and the result readily follows from (\ref{pull back K K_N}). 
\end{proof}
We deduce an extrinsic proof of the following well-known results:
\begin{corollary}
Assume that $M$ is a compact Lorentzian surface immersed in $\R^{2,2},$ such that $TM$ and $NM$ are oriented (in space and in time). Then
$$\int_MK\ \omega_M=0\hspace{.5cm}\mbox{and}\hspace{.5cm}\int_MK_N\ \omega_M=0,$$
where $\omega_M$ is the area form of $M.$
\end{corollary}
\begin{proof}
Since $*^2=id_{\Lambda^2\R^{2,2}}$ and similarly to the Euclidian case, we have the splitting
$$\Lambda^2\R^{2,2}=\Lambda^+\R^{2,2}\ \oplus\ \Lambda^-\R^{2,2},$$
where $\Lambda^+\R^{2,2}$ and $\Lambda^-\R^{2,2}$ are the eigenspaces of $*$ associated to the eigenvalues $+1$ and $-1$ respectively; these two 3-dimensional spaces are orthogonal, and equipped with a metric of signature (1,2). In this splitting 
\begin{equation}\label{splitting Q}
\mathcal{Q}=\mathcal{H}_1\times\mathcal{H}_2,
\end{equation}
where $\mathcal{H}_1$ and $\mathcal{H}_2$ are the hyperboloids 
$$\mathcal{H}_1=\{\eta\in\Lambda^+\R^{2,2}:\ \langle\eta,\eta\rangle=-1/2\}\hspace{.3cm}\mbox{and}\hspace{.3cm}\mathcal{H}_2=\{\eta\in\Lambda^-\R^{2,2}:\ \langle\eta,\eta\rangle=-1/2\}.$$
Let us write $G=(g_1,g_2)$ in the decomposition (\ref{splitting Q}). We have
$$G^*\omega_T=\frac{1}{2}\left(g_1^*\omega_1+g_2^*\omega_2\right)\hspace{.5cm}\mbox{and}\hspace{.5cm}G^*\omega_N=\frac{1}{2}\left(g_1^*\omega_1-g_2^*\omega_2\right),$$
where $\omega_1$ and $\omega_2$ are the 2-forms on $\mathcal{H}_1$ and $\mathcal{H}_2$ such that
$$[X,Y]={\omega_1}_{p_1}(X_1,Y_1)\ p_1+{\omega_2}_{p_2}(X_2,Y_2)\ p_2$$
for all $X=X_1+X_2$ and $Y=Y_1+Y_2\in T_p\mathcal{Q}\simeq T_{p_1}\mathcal{H}_1\oplus T_{p_2}\mathcal{H}_2$ ($\omega_1$ and $\omega_2$ are in fact the natural area forms on $\mathcal{H}_1$ and $\mathcal{H}_2$). Now, since $\mathcal{H}_1$ and $\mathcal{H}_2$ are not bounded, we necessarily have $\deg g_1=\deg g_2=0$ and 
$$\int_Mg_1^*\omega_1=\int_Mg_2^*\omega_2=0;$$
thus 
$$\int_M G^*\omega_T=\int_M G^*\omega_N=0,$$ 
and (\ref{pull back omega T N}) implies the result.
\end{proof}
We finish this section with an interpretation using the Gauss map of the vector $\mu_{II}$ and of the new invariants $\alpha$ and $\beta$ defined at a quasi-umbilic point of a Lorentzian surface $M,$ i.e. at a point $p$ where the second fundamental form is quasi-umbilic (Definition \ref{def quasi-umbilic}). First, for all unit vector $u$ belonging to $T_pM,$ if $u^{\perp}$ is a vector such that $u,u^{\perp}$ is a positively oriented Lorentzian basis of $T_pM,$
then
\begin{equation}\label{dG decomposition}
dG(u)=-\vec{H}\wedge u^{\perp}+II^0(u,u)\wedge u^{\perp}+u\wedge II^0(u,u^{\perp})
\end{equation}
where the traceless second fundamental form $II^0$ is given by
\begin{equation}\label{expr II^0}
II^0(e_1,e_1)=\pm\frac{1}{2}\mu_{II},\hspace{.5cm}II^0(e_2,e_2)=\pm\frac{1}{2}\mu_{II}\hspace{.5cm}\mbox{and}\hspace{.5cm}II^0(e_1,e_2)=\frac{1}{2}\mu_{II}
\end{equation}
(Lemma \ref{vec_dist}). We interpret each term in (\ref{dG decomposition}) as an infinitesimal rotation of the tangent plane in the direction $u:$ the first term $-\vec{H}\wedge u^{\perp}$ represents a \emph{mean infinitesimal rotation} of the tangent plane (the \textit{mean} is with respect to the tangent directions) in the hyperplane $T_pM\oplus\vec{H},$ around the tangent direction $u^{\perp}$ and with velocity $\vec H,$ whereas the term $II^0(u,u)\wedge u^{\perp}$ (resp. $u\wedge II^0(u,u^{\perp})$) represents an infinitesimal rotation of the tangent plane in the hyperplane $T_pM\oplus\R II^0(u,u)$ (resp. $T_pM\oplus\R II^0(u,u^{\perp}))$ around the tangent direction $u^{\perp}$ (resp. $u$), with velocity $II^0(u,u)$ (resp. $II^0(u,u^{\perp})$). Using (\ref{expr II^0}) we may easily get 
\begin{equation}\label{formulas II quasi-umbilic}
II^0(u,u)=-II^0(u,u^{\perp})=-\mu_{II}\langle u,N_1\rangle^2\hspace{.5cm}\mbox{or}\hspace{.5cm}II^0(u,u)=II^0(u,u^{\perp})=\mu_{II}\langle u,N_2\rangle^2
\end{equation}
depending the sign in (\ref{expr II^0}), where $N_1$ and $N_2$ are the null tangent vectors $\frac{\sqrt{2}}{2}(e_1+e_2)$ and $\frac{\sqrt{2}}{2}(e_2-e_1).$ In fact the formulas (\ref{formulas II quasi-umbilic}) characterize a quasi-umbilic point: the two infinitesimal rotations $II^0(u,u)\wedge u^{\perp}$ and $u\wedge II^0(u,u^{\perp})$ take place in the same hyperplane $T_pM\oplus\R\mu_{II},$ with the same velocities, proportional to the squared of the projection of the direction $u$ onto one of the two null lines of $T_pM.$ Finally the invariants $\alpha$ and $\beta$ determine the mean infinitesimal rotation once the vector $\mu_{II}$ is known.

\section{Asymptotic directions on a Lorentzian surface in $\R^{2,2}$}\label{section asymptotic directions}
In this section, we introduce the asymptotic directions of a Lorentzian surface in $\R^{2,2}$ by means of its Gauss map, give an intrinsic equation for the asymptotic lines on a Lorentzian surface, discuss their causal characters and show that the asymptotic directions correspond to directions of degeneracy of natural height functions defined on the surface. We then introduce the mean directionally curved directions on a Lorentzian surface in $\R^{2,2}$ and mention some of their relations with the asymptotic directions. We finally study the asymptotic directions of Lorentzian surfaces in Anti de Sitter space.

\subsection{Definition, intrinsic equation and causal character}
We still assume that $M$ is an oriented Lorentzian surface  in $\R^{2,2}$ and denote by $G:M\rightarrow\mathcal{Q}$ its Gauss map. Let us consider the quadratic map 
$$\delta: T_pM \rightarrow \Lambda^4\R^{2,2},\ \ \delta(v)=\frac{1}{2}dG(v) \wedge dG(v),$$ 
where $\Lambda^4\R^{2,2}$ is the space of 4-vectors of $\R^{2,2}$. Since $\Lambda^4\R^{2,2}$ naturally identifies to $\R$ (using the canonical volume element $e_1\wedge e_2\wedge e_3\wedge e_4$), $\delta$ may also be considered as a quadratic form on $T_pM.$ 
\begin{definition}\label{def asymptotic}
A non-zero vector $v \in T_pM$ defines an {\it asymptotic direction} at $p$ if $\delta(v)=0$.
\end{definition}
\begin{remark}
If $v,v'\in T_pM$ are such that  $G(p)=v\wedge v',$ then 
$$dG=II(v,.)\wedge v'+v\wedge II(v',.)$$ 
and 
\begin{equation}\label{delta v vp}
\delta(v)=v \wedge v'\wedge II(v,v')\wedge II(v,v).
\end{equation} 
Thus $v$ is an asymptotic direction if and only if $II(v,v')$ and $II(v,v)$ are linearly dependent.
\end{remark}
We analyze in detail the case where rank $f_{II}=2$ and the signature of $\Phi_{II}$ is $(1,1)$, assuming moreover that
\begin{equation}\label{simple hyp}
(|\H|^2-K)^2-K_N^2>0\hspace{.5cm}\mbox{and}\hspace{.5cm}|\H|^2-K>0.
\end{equation}
Let us first describe the second fundamental form in the basis of eigenvectors $(u_1,u_2)$ of $U_{\phi}$ given by Proposition $\ref{diag}$. The second hypothesis in (\ref{simple hyp}) implies that the eigenvalues $a$ and $b$ are positive, and we set $\mathsf{b}=\sqrt{b}$. The normal vector $\nu_0= \frac{1}{\mathsf{b}}u_2$ is such that $\Phi_{II}(\nu_0)=1$. Moreover, straightforward computations yield $\tilde \Phi_{II}(\nu_0,u_1)=0$ and $\tilde \Phi_{II}(\nu_0,u_2)=\mathsf{b}$ and thus 
$$\tilde \Phi_{II}(\nu_0,\nu)=\mathsf{b}\langle \nu,u_2\rangle$$ 
for all $\nu\in N_pM.$ On the other hand, $A_{II}(\nu_0,u_2)=0$ and Equation (\ref{rela_formas}) yields $A_{II}^2(\nu_0,u_1)=a;$ thus 
$A_{II}(\nu_0,\nu)= \mathsf{a}\langle \nu, u_1\rangle$ where $\mathsf{a}=\sqrt{a}$ or $-\sqrt{a}.$ Therefore, in some positively oriented and orthonormal basis $(e_1,e_2)$ of $T_pM,$
\begin{eqnarray}\label{e1}
S_{u_1} & = & - \alpha I \pm A_{II}(\nu_0,u_1)E_2=\left(\begin{array}{cc}-\alpha&\mp \mathsf{a}\\\pm \mathsf{a}&-\alpha\end{array}\right)
\end{eqnarray}
and
\begin{eqnarray}\label{e2}
S_{u_2} & = & \beta I \pm \tilde \Phi_{II}(\nu_0,u_2)E_1=\left(\begin{array}{cc}\beta\pm \mathsf{b}&0\\0&\beta\mp \mathsf{b}\end{array}\right)
\end{eqnarray}
(recall the normal form of $S_{\nu}$ in the table Section \ref{section reduction quadratic map}), and we get
\begin{equation}\label{segundaforma1}
II =  \left[\left(\begin{array}{cc}-\alpha&0\\0&\alpha\end{array}\right) \mp\left(\begin{array}{cc}0&\mathsf{a}\\\mathsf{a}&0\end{array}\right)\right]u_1+\left[\left(\begin{array}{cc}-\beta&0\\0&\beta\end{array}\right)\mp\left(\begin{array}{cc}\mathsf{b}&0\\0&\mathsf{b}\end{array}\right)\right]u_2
\end{equation}
(keeping in mind the relations $\langle II(X), u_i\rangle=\langle S_{u_i}(X),X \rangle,\ \ i=1,2,$ with $|u_1|^2=-|u_2|^2=-1$). Straightforward computations then give the classical invariants of the second fundamental form in terms of $\mathsf{a},$ $\mathsf{b},$ $\alpha$ and $\beta:$ we have
\begin{eqnarray}\label{relacionesentreinvariantes}
|\H|^2 & = & - \alpha^2+\beta^2,\hspace{1cm} K= - \alpha^2+\beta^2-\mathsf{a}^2-\mathsf{b}^2,\\ \nonumber 
\Delta & = & -\mathsf{a}^2\beta^2+ \mathsf{a}^2\mathsf{b}^2+ \alpha^2\mathsf{b}^2\hspace{.5cm}\mbox{and}\hspace{.5cm} K_N=2\mathsf{a}\mathsf{b}.
\end{eqnarray}
Further, since 
$$dG(e_1)=II(e_1,e_1)\wedge e_2+e_1\wedge II(e_2,e_1)$$
and
$$dG(e_2)=II(e_1,e_2)\wedge e_2+e_1\wedge II(e_2,e_2),$$
we easily get
$$\delta(e_1,e_1)=\pm \mathsf{a}(\beta\pm \mathsf{b}),\ \delta(e_2,e_2)=\pm \mathsf{a}(\beta\mp \mathsf{b})\hspace{.2cm}\mbox{and}\hspace{.2cm}\delta(e_1,e_2)=\mp \alpha \mathsf{b}.$$
Thus, if $v= x e_1 +y e_2,$
\begin{equation}\label{expr delta}
\delta(v)=\pm \mathsf{a}(\beta\pm \mathsf{b})x^2\pm \mathsf{a}(\beta\mp \mathsf{b})y^2\mp 2\alpha \mathsf{b} xy,
\end{equation}
which proves the following:
\begin{proposition}\label{prop asymptotic lines}
Assuming that (\ref{simple hyp}) holds, then, in a positively oriented and orthonormal basis $(e_1,e_2)$ of $T_pM$ such that 
\begin{equation}\label{basis adapted axes}
II^0(e_1)=II^0(e_2)=\mp\mathsf{b}u_2\hspace{.5cm}\mbox{and}\hspace{.5cm}II^0(e_1,e_2)=\mp \mathsf{a}u_1
\end{equation}
where $II^0$ is the traceless second fundamental form, the equation of the asymptotic directions is 
\begin{equation}\label{delta non qo}
\mathsf{a}(\beta\pm \mathsf{b})x^2+\mathsf{a}(\beta\mp \mathsf{b})y^2-2\alpha \mathsf{b} xy =0,
\end{equation}
where $\mathsf{a},$ $\mathsf{b},$ $\alpha$ and $\beta$ are numerical invariants satisfying (\ref{relacionesentreinvariantes}). 
\end{proposition}
\noindent We will give applications of this intrinsic equation below.
\begin{remark}\label{rmk interp basis}
 The conditions in (\ref{basis adapted axes}) have the following simple interpretation in terms of the curvature hyperbola: the vectors $e_1$ and $e_2$ appear to be the preimages by the map $v\mapsto II(v)/|v|^2$ of the points of the hyperbola belonging to the spacelike axis.
\end{remark}
We now discuss the causal character of the asymptotic directions. We consider
\begin{equation}\label{def delta^o}
\delta^o:=\delta-\frac{1}{2}\tr_g\delta\ g,
\end{equation}
the traceless part of the quadratic form $\delta.$ Using (\ref{expr delta}) and 
the relations (\ref{relacionesentreinvariantes}), we easily get
\begin{equation}\label{inv delta}
\tr_g\delta=-K_N\hspace{1cm}\mbox{and}\hspace{1cm}disc(\delta):=-\det{}_g\delta=-\Delta,
\end{equation}
and also
$$disc(\delta^o):=-\det{}_g{\delta^o}=\frac{1}{4}K_N^2-\Delta.$$
Contrasting with the cases of Riemannian and Lorentzian surfaces in 4-dimensional Minkowski space $\mathbb R^{3,1}$ \cite{b_s_1,b_s_3}, the existence of asymptotic lines at a point on the surface is equivalent here to the condition $\Delta \geq 0$ at this point. By (\ref{def delta^o}),
\begin{equation}\label{lignes asymptotiques delta^o}
\delta(u)=0\hspace{1cm}\mbox{if and only if}
\hspace{1cm}\delta^o(u)=\frac{1}{2}K_N|u|^2.
\end{equation}
The causal character of the asymptotic directions appears to depend on the signs of the forms $\delta$, $\delta^o$ and their discriminants. The results are similar to the case of the Lorentzian surfaces in 4-dimensional Minkowski space \cite[p. 1708]{b_s_3}, and we only briefly describe them below. There are two main cases, depending on $disc(\delta)$. Let us analyze only the case when $disc(\delta)<0,$ that is, when two distinct asymptotic directions are defined. We then divide the discussion in four cases, according to the sign of $\delta^o.$

\textit{First case: $disc(\delta^o)>0$:} if $\delta^o$ is positive (resp. negative), the solutions $u$ of (\ref{lignes asymptotiques delta^o}) are necessarily spacelike (resp. timelike) if $K_N>0,$ and timelike (resp. spacelike) if $K_N<0.$ 

\textit{Second case: $disc(\delta^o)<0$:} let us denote by $u_{\delta^o}$ the traceless symmetric operator of $T_pM$ associated to $\delta^o;$ we then have $|u_{\delta^o}|^2=-\det(u_{\delta^o})<0$ (recall Section \ref{section reduction quadratic map}) and, in some positively oriented and orthonormal basis $(e_1,e_2)$ of $T_{p}M,$ the matrix of $u_{\delta^o}$ reads
$$M(u_{\delta^o},(e_1,e_2))=\pm\sqrt{-|u_{\delta^o}|^2}\left(\begin{array}{cc}0&1\\-1&0\end{array}\right);$$
see Remark \ref{form_cano}. Writing $u=xe_1+ye_2,$ (\ref{lignes asymptotiques delta^o}) then reads
 \begin{equation}\label{lignes asymptotiques delta^o base}
\delta(u)=0\hspace{.5cm}\mbox{if and only if}\hspace{.5cm}\pm2\sqrt{-|u_{\delta^o}|^2}xy=K_N(x^2-y^2).
\end{equation}
Thus, if $u=xe_1+ye_2$ is a non trivial solution of $\delta(u)=0$, so is $\overline{u}:=-ye_1+xe_2.$ Observe that these solutions are necessarily spacelike or timelike, and that if one of them is spacelike, the other one is timelike; thus, one asymptotic direction is spacelike and the other one is timelike.

\textit{Third case: $disc(\delta^o)=0,$ $\delta^o\neq 0$:} we then have $|u_{\delta^o}|^2=-\det(u_{\delta^o})=0,$ and the kernel of $u_{\delta^o}$ is a null line in $T_{p}M;$ there is thus a unique lightlike line of solutions for the equations in (\ref{lignes asymptotiques delta^o}). The other independent solution is thus a timelike or  a spacelike line. But using (\ref{lignes asymptotiques delta^o}) again, if $\delta^o\geq 0$ (resp. $\delta^o\leq 0$) this solution is necessarily spacelike (resp. timelike) if $K_N>0$ and timelike (resp. spacelike) if $K_N<0.$

\textit{Fourth case: $\delta^o=0$:} then $\delta(u)=-K_N|u|^2,$ and $\delta(u)=0\Longleftrightarrow |u|^2=0.$ Note that $\vec{H}=0$ in that case: since $K_N\neq 0$ the point is not quasi-umbilic, and, by (\ref{expr delta}), $\delta^o(v)=\pm \mathsf{a}\beta(x^2+y^2)\mp 2\alpha\mathsf{b} xy.$  Since $\delta^o=0$ and $K_N=2\mathsf{a}\mathsf{b}\neq 0,$ we get $\alpha=\beta=0,$ i.e. $\vec{H}=0.$ 
\\

We describe the causal character of the asymptotic directions in the following table; in the first column appear the different possible values for the signature of $\delta^o.$ To simplify the presentation we suppose that $K_N\geq 0;$ if $K_N\leq 0,$ we just have to systematically exchange the words ``spacelike" and ``timelike" in the table. 

\vspace{.3cm}

\begin{tabular}{|c|c|c|}
\hline
$\mbox{signature of }\delta^o$
&$\begin{array}{c}disc(\delta)< 0\\\mbox{two distinct asymptotic}\\\mbox{directions which are}\end{array}$&$\begin{array}{c}disc(\delta)=0,\ \delta\neq 0\\\mbox{a double asymptotic}\\\mbox{direction which is}\end{array}$\\
\hline
(2,0)&
spacelike&spacelike\\
\hline
(0,2)&
timelike&timelike\\
\hline
(1,1)&$\mbox{1 spacelike - 1 timelike}$&$\mbox{Not possible}$\\
\hline
(1,0)&$\mbox{1 lightlike - 1 spacelike}$&$\begin{array}{c}\mbox{lightlike} \end{array}$\\
\hline
(0,1)&$\mbox{1 lightlike - 1 timelike}$&$\begin{array}{c}\mbox{lightlike} \end{array}$\\
\hline
(0,0)&$\begin{array}{c}\mbox{lightlike}\\\mbox{with }\vec{H}=0\end{array}$&\mbox{Not possible}\\
\hline
\end{tabular}

\vspace{.5cm}

We finish this section with a characterization of a quasi-umbilic point of a Lorentzian surface in terms of its asymptotic directions. This characterization is very similar to a result given in \cite{b_s_3}; since the proof is also very similar, we only state the result, and refer to \cite{b_s_3} for details:
\begin{theorem}
Assume that $p \in M$ is such that $\delta \neq 0.$ Then $p$ is a quasi-umbilic point if and only if there is a double lightlike asymptotic direction at $p$.  
\end{theorem}

\subsection{Asymptotic directions and height functions}

Let us define the family of height functions on a Lorentzian surface $M$ in $\mathbb R^{2,2}$ as
$$H:\ M \times \mathbb R^{2,2}  \rightarrow \mathbb R,\ \ H(p,\nu)=\langle p,\nu\rangle+\ c,$$
where $c \in \mathbb R$. The function $h_\nu: M \rightarrow \mathbb R$ defined as $h_\nu= H(\cdot,\nu)$ is singular at $p \in M$, that is $dh_{{\nu}_p}=0$, if and only if $\nu$ is normal to $M$ at $p$. 
Consider also
$$Hess\ h_{\nu}:=\nabla dh_{\nu},$$
the Hessian of $h_{\nu},$ where $\nabla$ is here the Levi-Civita connection of $M$ acting on the 1-forms. We readily get that
\begin{equation}\label{relation h ff}
Hess\ h_{\nu}=II_{\nu}.
\end{equation}
We say that a non-zero normal vector $\nu$ at $p$ is a {\it binormal} vector if the quadratic form $Hess\ h_{\nu}$ is degenerate at $p$, and that a non-zero vector 
$v \in T_pM$ defines a {\it contact direction} if it belongs to the kernel of $Hess\ h_{\nu}$ at $p.$ 
Thus, by definition, $v$ is a contact direction with associated binormal vector $\nu$ if and only if the contact at $p$ between the surface and the hyperplane $\nu^{\perp}$ is of order $\geq 2$ in the direction $v.$ 

We now prove that $v\in T_pM$ is a contact direction if and only if it is an asymptotic direction. By (\ref{relation h ff}), we readily get the following result:
\begin{lemma}\label{zerocurvature}
A non-zero vector $v \in T_pM$ defines a contact direction if and only if  $S_{\nu}(v)=0$ for some non-zero vector $\nu$ normal to $M$ at $p$, where $S_{\nu}$ is the symmetric operator associated to the form $II_{\nu}$.  
\end{lemma}

Observe that the normal vector $\nu$ given by the lemma is a binormal vector with associated contact direction $v$.

\begin{proposition}\label{lineasdecontacto}
A vector $v \in T_pM$ defines a contact direction if and only if it defines an asymptotic direction.
\end{proposition}

\begin{proof}
Recalling (\ref{delta v vp}), $\delta(v)= 0$ if and only if $II(v,v')$ and $II(v,v)$ are linearly dependent, that is, if and only if the linear map $II(v,\cdot):T_pM \rightarrow N_pM$ 
has a non trivial kernel; this is equivalent to the existence of a non trivial vector $\nu \in N_pM$ normal to the image of this map, i.e. such that $\langle II(v,\cdot), \nu\rangle=0.$ Since 
\begin{eqnarray*}
\langle II(v,w), \nu\rangle= \langle S_{\nu}(v), w\rangle
\end{eqnarray*}   
for all $w \in T_pM,$ we conclude that $v$ defines an asymptotic direction if and only if $S_{\nu}(v)=0$ for some non-zero normal vector $\nu.$  By Lemma \ref{zerocurvature} this also characterizes a contact direction.
\end{proof}
\begin{remark} 
The notion of contact direction has been used before in different settings, see for instance \cite{C-M-RF}, \cite{b_s_3}. It is usually used to define the notion of asymptotic direction. In this paper, we rather defined the asymptotic directions by means of the Gauss map, and finally proved that the two notions coincide.
\end{remark}

\subsection{Asymptotic and mean directionally curved lines}

Now let us analyze {\it the mean directionally curved field of directions}, studied for surfaces immersed in $\mathbb R^4$ in \cite{Mello} and \cite{Tari}, and for timelike surfaces in Minkowski space $\mathbb R^{3,1}$ in \cite{b_s_3}. In $\R^{2,2}$ these directions are defined as the pull-back by the second fundamental form of the intersection points in the normal plane of the curvature hyperbola with the line generated by the mean curvature vector. More precisely, the condition is \begin{eqnarray}\label{Meancurveddirections}
[\vec H, II(v)]=0,
\end{eqnarray}  
where the brackets stand for the determinant of the vectors in a positively oriented and orthonormal basis of the normal plane. This is also $[\vec H, II^0(v)]=0,$ where $II^0$ is the traceless part of the second fundamental form. For sake of simplicity, here again we assume that 
$$(|\H|^2-K)^2-K_N^2>0\hspace{.5cm}\mbox{and}\hspace{.5cm} |\H|^2-K>0;$$ 
under these assumptions, in a positively oriented and orthonormal basis $(e_1,e_2)$ of $T_pM$ (see Proposition \ref{prop asymptotic lines} and Remark \ref{rmk interp basis} above), the second fundamental form is given by (\ref{segundaforma1}), and (\ref{Meancurveddirections}) reads
\begin{eqnarray*}
[\H, \mathsf{a}(2xy)u_1+\mathsf{b}(x^2 + y^2) u_2]=0.
\end{eqnarray*}
Thus, we obtain the following intrinsic equation of these directions:
\begin{proposition}
In $(e_1,e_2),$ the equation of the mean directionally curved directions in terms of the invariants $\mathsf{a},$ $\mathsf{b}$, $\alpha$ and $\beta$ is
\begin{eqnarray}
\alpha \mathsf{b}(x^2+y^2)-2\mathsf{a}\beta xy=0.
\end{eqnarray}
\end{proposition}
Moreover, using this equation and the expression (\ref{expr delta}) of $\delta,$ we deduce the following:
\begin{lemma}
Equation $(\ref{Meancurveddirections})$ is equivalent to
\begin{equation}\label{equation mean delta}
\delta(v,v^*)=0,
\end{equation}
where $v= x e_1 + y e_2$ and $v^*= y e_1 + x e_2$ in a positively oriented and orthonormal basis $(e_1,e_2)$ of $T_pM$. 
\end {lemma} 
\begin{corollary}
Under the hypotheses above, the mean directionally curved directions bisect the asymptotic directions.
\end{corollary}
\begin{proof}
Let $v,v^*$ be the mean directionally curved directions. Assuming moreover that 
$$|v|^2=-|v^*|^2=\pm 1,$$ 
the vectors $v$ and $v^*$ form a Lorentzian basis of $T_pM;$ then, a unit direction
$$v_\psi:=\cosh\psi\ v+\sinh\psi\ v^*,\hspace{.5cm}\psi\in\R$$
is an asymptotic direction if and only if
$$\delta(v_{\psi})=\cosh^2\psi\ \delta(v)+\sinh^2\psi\ \delta(v^*)=0$$
(Definition \ref{def asymptotic} and Equation (\ref{equation mean delta})). Thus $v_{\psi}$ is an asymptotic direction if and only if so is $v_{-\psi},$ which gives the result.
\end{proof}

\subsection{Asymptotic directions on Lorentzian surfaces in Anti de Sitter space}

Let us apply these results to the analysis of the Lorentzian surfaces immersed in the {\it Anti de Sitter 3-space}. This space is defined by
$${\mathbb H}_1^3=\{ x \in \mathbb R^{2,2}:\ \langle x,x\rangle=-1 \}.$$
It is the 3-dimensional Lorentzian space form with negative curvature. The geometry of Lorentzian surfaces in this space has been studied with an approach of Singularities 
by analyzing the contacts of the surfaces with some models \cite{Ch-I}. 
Following \cite{I-Y}, we consider the {\it $\phi-$de Sitter height function} defined on a Lorentzian surface $M$ in ${\mathbb H}_1^3$ by
$$H_{\phi}: M \times  S^3_2(\sin^2 \phi) \rightarrow \mathbb R,\ \ H_{\phi}(p,\nu)=\langle p,\nu\rangle+\ \cos \phi,\ \ \phi \in [0, \pi/2],$$
where $S^3_2(\sin^2 \phi):=\{x \in \mathbb R^{2,2}:\ \langle x,x\rangle=\sin^2 \phi \}$ is the pseudo sphere with index $2$ centered at the origin and with radius $\sin^2 \phi$ if $\phi \neq 0;$ if $\phi=0$ this set is the null cone at the origin $\{x \in \mathbb R^{2,2}:\ \langle x,x\rangle=0 \}$. 

Let $\varphi: U \rightarrow  {\mathbb H}_1^3$ be an immersion of an open set $U \subset \mathbb R^2$ 
with coordinates $u=(u_1,u_2)$, whose image $M= \varphi(U)$ is a Lorentzian surface. The vector field  
$$N(u)=\frac{\varphi(u)\wedge \varphi_{u_1}(u)\wedge \varphi_{u_2}(u)}{| \varphi(u)\wedge \varphi_{u_1}(u)\wedge \varphi_{u_2}(u) |}$$
is by definition unitary, normal to $M$ and tangent to  ${\mathbb H}_1^3,$ i.e. is such that $(\varphi_{u_1}, \varphi_{u_2}, \varphi, N)$ is a frame on $M$ whose first two vectors generate the tangent bundle and the last two vectors the normal bundle of $M$ in $\R^{2,2}.$ The $\phi_{\pm}$- de Sitter duals of $M$ are defined as
\begin{eqnarray*}
N_{\pm}^{\phi}: U \rightarrow S^3_2(\sin^2 \phi),\ \ \ N_{\pm}^{\phi}(u)= \cos\phi\ \varphi(u) \pm  N(u).  
\end{eqnarray*} 
Since $N_+^{\phi}(u)= -N_-^{\phi + \phi/2}(u)$ we only consider $\phi \in I= [0, \pi/2]$.
 
The family of height functions $(H_{\phi})_{\phi\in I}$ is a generating family of a natural Legendrian embedding of the surface into a contact manifold $\Delta_{21}^{\pm}$ whose structure is similar to that defined in \cite{I-Y}. Moreover, the image of the $\phi_{\pm}$- de Sitter dual is the wave front set of this Legendrian map . Furthermore, the fields $N_{\pm}^{\phi}$ are normal to $M$ and the $\phi_{\pm}$-Gauss-Kronecker curvature at each point of $M$ is defined as the determinant of the linear operator $-dN_{\pm}^{\phi}:T_pM \rightarrow T_pM$. A point $p$ where this curvature vanishes is called a {\it $N_{\pm}^{\phi}$-parabolic} point; such a point is characterized as a point where the normal $N_{\pm}^{\phi}$ is a binormal vector.

The results proved at the beginning of the section imply a rigidity property of the contact directions associated to the different binormal de Sitter duals $N^{\phi}_{\pm}$ parameterized by $\phi$. Indeed,  if $N^{\phi}_{+}$ is a family of binormal vectors at $p$ parameterized by $\phi \in I,$ then there is a family of contact directions $v_{\phi}$ associated to the corresponding family of height functions. 
If the discriminant of the form $\delta$ satisfies $\Delta_p=0$, there is only one asymptotic direction, and Proposition $\ref{lineasdecontacto}$ implies that $v_{\phi}$ coincides with it for any $\phi$. If $\Delta_p>0$, there are two asymptotic lines at $p,$ $l_1$ and $l_2$ say. We assert that $l_1$ (or $l_2$) is a contact direction of one of the two families of binormals $(N^{\phi}_{+})_{\phi \in [0,2\pi]}$ or $(N^{\phi}_{-})_{\phi \in [0,2\pi]},$ that is, is a contact direction for the binormals $N^{\phi}_{+}$ \emph{for all $\phi \in [0,2\pi]$}, or is a contact direction for the binormals $N^{\phi}_{-}$ \emph{for all $\phi \in [0,2\pi]$}. Indeed, let  $v_0$ be the contact direction of the height function defined by the binormal $N^{\phi_0}_{+}$ for some $\phi_0 \in I,$ and let $(\phi_k)_{k \in \mathbb N}$ be a real sequence converging to $\phi_0;$ we can choose, associated to each shape operator of the sequence $(S_{\phi_k})_{k \in \mathbb N}$ defined by the binormals  $(N^{\phi_k}_{+})_{k \in \mathbb N}$, an eigenvector $v_k$ corresponding to its null eigenvalue and such that the sequence $v_k$ converges to $v_0.$ The possibility of such a continuous choice implies the following:

\begin{proposition}
Let $p$ be a $N^{\phi}_+$-parabolic point on $M$ for all $\phi \in I$.
Then, the contact directions of the height functions defined by the binormal vectors $N^{\phi}_+$, parameterized by $\phi \in I$ at $p$ coincide.   
\end{proposition} 

\section{Quasi-umbilic surfaces in $\R^{2,2}$}\label{section quasi umbilic}
\subsection{Description of the quasi-umbilic surfaces}\label{subsection quasi umbilic}
Quasi-umbilic (Lorentzian) surfaces in 3 and 4-dimensional Minkowski space were described in \cite{c} and \cite{b_s_3} respectively. We are interested here in quasi-umbilic surfaces in $\R^{2,2}:$ similarly to \cite{b_s_3}, we will say that a Lorentzian surface $M$ in $\R^{2,2}$ is quasi-umbilic if its second fundamental form is quasi-umbilic at every point of $M,$ which means that the curvature hyperbola degenerates to a straight line with one point removed at every point of $M$, or equivalently that  
\begin{equation}\label{inv_casi-umb} 
|\vec{H}|^2=K \hspace{0.2in}\text{and}\hspace{0.2in}K_N=\Delta=0 
\end{equation} 
together with 
\begin{equation}
\Phi_{II}=0\hspace{1cm}\mbox{and}\hspace{1cm}II\neq\vec{H}g 
\end{equation}
on $M;$ see Proposition \ref{imagen_fq_2} \textit{1}- above. Similarly to \cite[Theorem 5.1]{b_s_3}, the quasi-umbilic surfaces in $\R^{2,2}$ are described as follows:
\begin{theorem}\label{sup_umb_casi-umb} A Lorentzian surface $M$ in $\R^{2,2}$ is umbilic or quasi-umbilic if and only if it is parameterized by \begin{equation}\label{paramet} \psi(s,t)=\gamma(s)+tT(s)\end{equation}
where $\gamma$ is a lightlike curve in $\R^{2,2}$ and $T$ is some lightlike vector field along $\gamma$ such that  $\gamma'(s)$ and $T(s)$ are independent for all value of $s.$
\end{theorem}  
This result generalizes the main result of \cite{c} to the space $\R^{2,2}.$ We omit the proof since it is identical to the proof of Theorem 5.1 in \cite{b_s_3} (note that a lemma similar to  the key lemma \cite[Lemma 5.2]{b_s_3} is also valid here). 
\begin{remark}
To our knowledge, the natural problem of the description of the Lorentzian surfaces in $\R^{2,2}$ which are umbilic at every point is an open question (note that Lorentzian umbilic surfaces in 4-dimensional Minkowski space are well-known, see e.g. \cite{Hong}). 
\end{remark}
\begin{remark}
It may occur that (\ref{inv_casi-umb}) holds, but with $\Phi_{II}\neq 0$ (Proposition \ref{imagen_fq_1}, \textit{1}-$(b)$ or \textit{2}-$(b),$ with $\Delta=0$). In that case, we have in fact  $|\vec{H}|^2=K=K_N=\Delta=0$: indeed, we are in the context of Proposition \ref{nodia} 2., with
$$M(U_{\Phi},(u_1,u_2))=\begin{pmatrix}\epsilon_1&\epsilon_2\\-\epsilon_2&-\epsilon_1\end{pmatrix},\hspace{1cm}\epsilon_1=\pm1,\ \epsilon_2=\pm1$$ 
($(u_1,u_2)$ is the positively oriented and orthonormal basis of $NM$ given by the proposition); more precisely, writing $\vec{H}:=\alpha u_1+\beta u_2$ the second fundamental form is in fact given by 	
\begin{enumerate}
	\item $\epsilon_1=-1:$ \hspace{0.2in}  $II=\begin{pmatrix}-\alpha\pm1&0 \\0 & \alpha\pm1 \end{pmatrix}u_1+\begin{pmatrix}-\beta \pm\epsilon_2& 0\\ 0& \beta\pm\epsilon_2\end{pmatrix}u_2,$ which implies that 
	$$\Delta=(\alpha-\epsilon_2\beta)^2\geq0,$$
		\item $\epsilon_1=1:$ \hspace{0.2in} $II=\begin{pmatrix}-\alpha&\mp1\\ \mp1& \alpha\end{pmatrix}u_1+\begin{pmatrix}-\beta & \pm\epsilon_2\\ \pm\epsilon_2& \beta\end{pmatrix}u_2,$ which implies that 
		$$\Delta=-(\alpha+\epsilon_2\beta)^2\leq0;$$  
	\end{enumerate}
see the table in Section \ref{section reduction quadratic map}. Since $\Delta=0,$ we deduce that $\alpha=\pm\beta$ i.e. $|\vec{H}|^2=0.$ If the Gauss map of the surface is regular, the surface belongs in fact to a degenerate hyperplane (see Theorem \ref{surfaces invariants zero} below).
\end{remark}

\subsection{Lorentzian surfaces such that $|\H|^2=K=K_N=\Delta=0$}
We describe here the Lorentzian surfaces in $\R^{2,2}$ whose classical invariants are zero. We will say that an hyperplane of $\R^{2,2}$ is \emph{degenerate} if the metric of $\R^{2,2}$ induces on it a degenerate metric. We state the main result of the section:
\begin{theorem}\label{surfaces invariants zero}
Let $M$ be an oriented Lorentzian surface  in $\R^{2,2}$ with regular Gauss map and such that $K=K_N=\Delta=0$. Then
\begin{enumerate}
\item if $\Phi_{II}\neq 0,$ $M$ belongs to a degenerate hyperplane;
\item if $\Phi_{II}\equiv 0,$ $M$ is a flat umbilic or quasi-umbilic surface.
\end{enumerate}
In both cases, we have in fact
$$|\H|^2=K=K_N=\Delta=0.$$
Conversely, if $M$ belongs to a degenerate hyperplane or is a flat umbilic or quasi-umbilic surface then $|\H|^2=K=K_N=\Delta=0.$ 
\end{theorem}
\begin{proof}
We assume that $M$ satisfies the hypotheses of the theorem, and that $\Phi_{II}\neq 0$ (if $\Phi_{II}\equiv 0,$ $M$ is umbilic or quasi-umbilic by the very definition). The quadratic form $\Phi_{II}$ is degenerate since $K_N=0$. We first prove by contradiction that $U_{\Phi}$ is not diagonalizable; we assume that it is diagonalizable, and consider two cases:

\emph{1- $\Phi$ has signature (1,0): } $U_{\Phi}$ is then given by (\ref{diag2}) with $a=|\H|^2, b=0$ if $|\H|^2<0,$ or $a=0, b=|\H|^2$ if $|\H|^2>0$ (note that $U_{\Phi}^0\neq 0$ since $\Phi_{II}$ is degenerate and not zero); $\beta=0$ i.e. $\H=\alpha u_1$ in the first case, and $\alpha=0$ i.e. $\H=\beta u_2$ in the second case (formulas (\ref{inva1})-(\ref{inva2})). The curvature hyperbolas are given by Proposition \ref{imagen_fq_1} 1- (a), and, in each case, the vector $0\in NM$ appears to be an extremal point of the (degenerate) hyperbola: if $u\in TM,\ |u|^2=\pm 1$ is such that $II(u,u)=0,$ we thus also have $II(u,v)=0$ for all $v\in TM,$ that is $dG(u)=0,$ a contradiction with the hypothesis that $G$ is regular. 

\emph{2- $\Phi$ has signature (0,1): } we then have  $a=|\H|^2, b=0$ if $|\H|^2>0,$ or $a=0, b=|\H|^2$ if $|\H|^2<0$ in Proposition \ref{diag}, and formulas (\ref{inva1})-(\ref{inva2}) give $\alpha^2=-|\H|^2,$ $\beta^2=0$ in the first case, and $\alpha^2=0,$ $\beta^2=|\H|^2$ in the second case; this is not possible since $\alpha^2$ and $\beta^2$ are necessarily non-negative. 

Thus $U_{\Phi}$ is not diagonalizable, and $\H$ is zero or lightlike (by conditions (\ref{cond uphi nodiag}) in Proposition \ref{nodia}). Recalling the normal forms in the table Section \ref{section reduction quadratic map}, we have
$$S_{\nu}=\langle \H,\nu\rangle\left(\begin{array}{cc}1&0\\0&1\end{array}\right)\pm\tilde\Phi_{II}(\nu_0,\nu)E_i$$ 
in some positive oriented basis $e_1,e_2$ of $TM,$ where $\nu_0$ is a vector belonging to $NM$ and $E_i=E_1$ or $E_2,$ that is
\begin{equation}\label{expr II proof}
II=\H\left(\begin{array}{rr}-1&0\\0&1\end{array}\right)\pm U_{\Phi}(\nu_0)\tilde{E}_i,
\end{equation}
where $\tilde{E}_1=\left(\begin{array}{rr}-1&0\\0&-1\end{array}\right)$ and $\tilde{E}_2=\left(\begin{array}{rr}0&-1\\-1&0\end{array}\right).$ 
Thus
$$II(e_1,e_1)=-\H-\varepsilon U_{\Phi}(\nu_0),\ II(e_2,e_2)=\H-\varepsilon U_{\Phi}(\nu_0)\hspace{.3cm}\mbox{and}\hspace{.3cm} II(e_1,e_2)=0$$
in the first case, and
$$II(e_1,e_1)=-\H,\ II(e_2,e_2)=\H\hspace{.3cm}\mbox{and}\hspace{.3cm} II(e_1,e_2)=-\varepsilon U_{\Phi}(\nu_0)$$
in the second case, where $\varepsilon=\pm 1$. Using that $dG=II(e_1,.)\wedge e_2+e_1\wedge II(e_2,.),$ we then compute the matrix of $\delta:=\frac{1}{2}dG\wedge dG$ in $e_1,e_2:$ it is of the form $\left(\begin{array}{rr}0&c\\c&0\end{array}\right)$ in the first case and $\left(\begin{array}{rr}c&0\\0&c\end{array}\right)$ in the second case. Recalling (\ref{inv delta}), we have $\det_g\delta=\Delta=0,$ from which we get $c= 0,$ that is $\delta=0$ in both cases: since $G$ is moreover assumed to be regular, the surface necessarily belongs to a hyperplane (see \cite[Theorem 1.3]{L}, in the Euclidian context). This hyperplane is degenerate: this is clear if $\H\neq 0$ since $\H$ is then a non-zero lightlike vector, normal to the surface and belonging to the hyperplane; if now $\H=0,$ then (\ref{expr II proof}) reads $II=\pm U_{\Phi}(\nu_0)\tilde{E}_i,$ which gives
$$K=-\langle II(e_1,e_1),II(e_2,e_2)\rangle+|II(e_1,e_2)|^2=\pm |U_{\Phi}(\nu_0)|^2,$$
and, since $\Phi_{II}\neq 0$ and $K=0,$ the vector $U_{\Phi}(\nu_0)$ is non-zero, lightlike, normal to the surface and necessarily belongs to the hyperplane since it is in the range of $II$; the hyperplane is thus also degenerate in that case.

The converse statement readily follows from (\ref{inv_casi-umb}).
\end{proof}
\begin{remark} 
According to Theorem \ref{surfaces invariants zero} and the previous sections, the numerical invariants of a Lorentzian surface in $\R^{2,2}$ with regular Gauss map and whose classical invariants $|\H|^2,$ $K,$ $K_N$ and $\Delta$ all vanish are the invariants given by (\ref{def invariants Phi neq 0}) if $\Phi_{II}\neq 0$ and the invariants given by (\ref{nuev_inva}) if $\Phi_{II}=0$ and $f_{II}\neq 0$ (the quasi-umbilic case); there is no invariant if $f_{II}=0$ (the umbilic case). 
\end{remark}
\begin{remark}\label{exampleinvzero}
It is straightforward to check that the quasi-umbilic surface 
\begin{equation*}
\psi(s,t)=\gamma(s)+tT(s)
\end{equation*}
with
$$\gamma(s)=(a(s),-a(s),b(s),-b(s))$$
and 
$$T(s)=(f(s),f(s),g(s),g(s)),$$
where $a,b,f$ and $g$ are real functions of the variable $s$ such that
$$a'f+b'g\neq 0,\hspace{1cm} f'g-g'f\neq 0$$
and
\begin{equation}\label{hypothesisab}
b''a'-a''b'\neq 0,
\end{equation}
is such that $|\H|^2=K=K_N=\Delta=0,$ has regular Gauss map and does not belong to any hyperplane. If we assume that 
$$(b''a'-a''b')(s_0)=0\hspace{.5cm}\mbox{and}\hspace{.5cm}(b''a'-a''b')(s)\neq 0\ \mbox{for}\ s\neq s_0$$
instead of (\ref{hypothesisab}), we obtain a surface such that $|\H|^2=K=K_N=\Delta=0$ and with regular Gauss map, which is umbilic at $\psi(s_0,0)$ and quasi-umbilic at $\psi(s,t),$ $s\neq s_0.$
\end{remark}
\noindent
\textbf{Acknowledgement:} Sections 1, 2, 3 and 4.1 of this paper is part of V. Patty's PhD thesis; he thanks CONACYT for support.
The third author was partially supported by PAPIIT-DGAPA-UNAM grant No. 117714.

\end{document}